\documentclass[reqno]{amsart}
\usepackage{amsmath, amssymb, amsthm, amsfonts, amstext}
\usepackage{xcolor}
\usepackage[normalem]{ulem}
\usepackage[colorlinks=true, citecolor=red, linkcolor=blue]{hyperref}
\usepackage[T1]{fontenc}
\usepackage[utf8]{inputenc}
\usepackage[all]{xy}
\usepackage{tikz-cd}
\usepackage{enumerate,stmaryrd}
\usepackage{pdfpages}
\usepackage{appendix}
\usepackage{subfiles}
\usepackage{multirow}
\theoremstyle{plain}
\newtheorem{theorem}{\sc Theorem}[section]

\newtheorem{lemma}[theorem]{\sc Lemma}

\newtheorem{proposition}[theorem]{\sc Proposition}
\newtheorem{corollary}[theorem]{\sc Corollary}

\theoremstyle{definition}
\newtheorem{definition}[theorem]{\sc Definition}
\newtheorem{remark}[theorem]{\sc Remark}

\newtheorem*{thmA}{Theorem A}
\newtheorem*{thmB}{Theorem B}
\newtheorem*{thmC}{Theorem C}
\newtheorem*{thmD}{Theorem D}

\numberwithin{equation}{section}
\newcommand{\ap}{{a^{\psi}}}
\newcommand{\app}{{a^{\psi_2}}}
\newcommand{\bp}{{b^{\psi}}}
\newcommand{\bpp}{{b^{\psi_2}}}
\newcommand{\Id}{\operatorname{I}}
\title{On weak commutativity between $n$ isomorphic groups, $n \geq 2$}
\author[Sidki]{S.\, N. Sidki } 
\address{Departamento de Matemática - IE, Universidade de Brasília, Brasília - DF, Brazil }
\email{ssidki@gmail.com}
\author[Rocco]{N.\, R. Rocco }
\address{Departamento de Matemática - IE, Universidade de Brasília, Brasília - DF, Brazil }
\email{norai.rocco@gmail.com}
\author[Oliveira]{R. N. de Oliveira}
\address{Instituto de Matemática e Estatística, Universidade Federal de Goiás - Goiânia - GO, Brazil  }
\email{ricardo@ufg.br}
\dedicatory{With an appendix by Dmytro Savchuk}
\begin{document}
\begin{abstract}
The Weak Commutativity Group was defined by the first author in the early 1980's as
\[ \chi(H) = \langle H, H^\psi \mid [h,h^\psi] = 1, \ \forall h \in H \rangle \]
where $H$ and $H^\psi$ are isomorphic groups under an isomorphism $\psi$, and proved that $\chi(-)$, seen as a group operator, preserved the classes of solvable groups, finite groups, finite $p$-groups, and perfect groups.  The list has grown considerably since then (see references). 

It is clear that the iterated operator $\chi^n (-)$ also preserves the same classes. However, we observe that for $n \geq 2$ many of the $2^n$ $H$-generating copies of $\chi^n (H)$ fail to weakly commute among themselves. 

We introduce in this paper the ``$n$-ary Weak Commutativity Group'' generated by $n$ copies of $H$ which weakly commute among themselves; that is
\[ \chi_n (H) = \langle H, H^\psi, \dots, H^{\psi_{n-1}} \mid [h^{\psi_i}, h^{\psi_j}] = 1, \ \forall h \in H, \ 0 \leq i < j \leq n-1 \rangle \]
(by definition, $H^{\psi_0} = H, H^{\psi_1} = H^\psi$).

We study quotients of $\chi^n (H)$ which are also simultaneously quotients of $\chi_{2^n} (H)$; we refer to them as hybrid $\chi^n (H)$ - $\chi_{2^n} (H)$ groups.

For $n=2$ we determine the largest common quotient $Q(H)$ of $\chi^2(H)$ and $\chi_4(H)$, that is, the largest hybrid $\chi^2(H)$--$\chi_4(H)$ group.

A related class of groups in our study is $\chi(n,H)$, a quotient of $\chi_n(H)$, previously introduced by the third named author in his 2007 doctoral thesis.  We prove that this class of groups continues to preserve finiteness, nilpotency and perfectness, and we provide detailed information about their structure.
In contrast, $\chi_n(H)$ behaves differently: $\chi_3(C_2 \times C_2)$ is infinite, while $\chi_3(C_{p^t} \times C_{p^u})$ is finite for $p$ odd.
\end{abstract}
\maketitle
\section{Introduction}

In this article, we use standard notation of group theory (see, for instance, \cite{Robinson}). In particular, for elements $x, y$ of a group $G$ we write $x^y$ for the conjugate $y^{-1}xy$,  $[x, y]$ for the commutator $x^{-1}y^{-1} x y$, and our commutators are left normed, $[x, y, z] := [[x, y], z].$ If $\alpha: G \to H$ is any group homomorphism, the image of $x \in G$ by $\alpha$ is written $x^\alpha$, except in some situations to avoid notational overload; if $\alpha$ is an endomorphism of $G$, we write $[g, \alpha]$ for the element $g^{-1}g^\alpha.$ Consequently, if $Y \subseteq G$ is any subset then $[Y, \alpha]$ denotes the subgroup of $G$ generated by all elements $[y, \alpha]$ with $y \in Y.$ 

For the sake of reference, we recall that $\chi$  is defined as an operator for any given group $H$ by
\begin{equation} \label{eq:wp}
\chi(H):= \langle H, H^\psi \;\mid\; [h, h^\psi] = 1, \ \forall h \in H \rangle, 
\end{equation}
where $H^\psi$ denotes an isomorphic copy of $H$ via the isomorphism $\psi: H \to H^\psi$, $h \mapsto h^\psi$ for all $h \in H$. For further developments on $\chi(H)$ and related constructions, see \cite{BD,EK,LO,BL-SS,os-cfg,rocco1,rocco-crnts,rocco-pcefsg}. 

Now, for a given group $H$, consider $n$ isomorphic copies of $H$ through the isomorphisms  
$\psi_i : H \longrightarrow H^{\psi_i}, \ 0 \leq i \leq n-1,$
with $\psi_0 = \Id_H$ and $\psi_1 = \psi$. We define the following groups: the iterated group $\chi^{n}(H)$, $\chi_n(H)$ (to which we refer affectionately as the mother group), and its quotient $\chi(n,H)$ by the following  presentations:
\begin{align}
\chi^{0}(H) &:= H, \qquad \chi^{n}(H) := \chi\big(\chi^{n-1}(H)\big), \ n \geq 1, \label{def:iter_grp}\\
\chi_n(H) &:= \big\langle H, H^{\psi}, \ldots, H^{\psi_{n-1}} \;\big|\; [h^{\psi_i}, h^{\psi_j}] = 1, \ \forall h \in H, \ i \neq j \big\rangle, \label{def:mother_grp}\\
\chi(n,H) &:= \big\langle\, \chi_n(H) \;\big|\; [D_{ij}, L_{kl}] = 1, \ \forall\, i \neq j, \ k \neq l \,\big\rangle, \label{def:chin_1}
\end{align}
where
$ D_{ij} : = [H^{{\psi}_i}, H^{{\psi}_j}], \bigl[ h^{\psi_i}, \psi_i^{-1}\psi_j \bigr] :=  h^{-\psi_i} \big(h^{\psi_i}\big)^{\psi_i^{-1}\psi_j} = h^{-\psi_i} h^{\psi_j}$, and $ L_{ij} := [H^{{\psi}_i}, \psi_i^{-1}\psi_j]$.

In this way the group $\chi_n(H)$ has $n$ isomorphic subgroups $H, \, H^{\psi},\, \ldots\,
, H^{{\psi}_{n-1}}$, that commute weakly with each other according to (\ref{eq:wp}).
Notice that the relations $[D_{ij}, L_{ij}] = 1$ for all pairs $(i,j), \; 0 \leqslant i, j \leqslant n-1, \; i \neq j,$ in the presentation of $\chi_n(H)$ are already consequences of the relations $[h^{\psi_i}, h^{\psi_j}]=1$ (see \cite[Proposition 4.1.7]{sidki-wp}). It is precisely the further relations $[D_{ij}, L_{kl}] = 1$ with $\{i,j\} \neq \{k,l\}$, which are not consequences of weak commutativity, that distinguish $\chi(n,H)$ from the mother group $\chi_n(H)$ and are responsible for preserving finiteness. The class $\chi(n,H)$ was modeled on $\chi(H)$ so as to preserve especially finiteness and nilpotency. On the other hand, it is expected that the mother group $\chi_n(H)$ would be infinite for finite non-cyclic groups $H$ and $n \geq 4$. Though, for $n = 3$ and small abelian groups the situation is mixed, as seen in Theorem~\hyperref[th:thmD]{D}: indeed, $\chi_3(C_2 \times C_2)$ is already infinite, while $\chi_3(H)$ remains finite when $H = C_{p^t} \times C_{p^u}$ with $p$ odd.  
 
Under the assumption that $H$ is a nilpotent group, the third named author obtained in \cite{oliveira-thesis} some upper bounds on the class of nilpotency of $\chi(n, H)$, similar to those found in \cite{grs}. Furthermore, an estimate was established for the order of $\chi(n,H)$ when $H$ is a finite abelian group. In this paper, we provide a new and sharper estimate for the order of $\chi(n, H)$ for any finite group $H$. Specifically, we prove 

\begin{thmA} \label{th:thmA} 
\textit{If $H$ is a finite group, then $|\chi(n,H)|$
divides $|\chi(H)|^{\binom{n}{2}} |H|^{n-2\binom{n}{2}}$}.
\
\end{thmA}

Such an estimate is attained in particular when $H$ is an elementary abelian $p$-group, with $p$ an odd prime; moreover, for $p = 2$ our estimate is even sharper.

If $H$ is a perfect group, then $\chi(n,H)$ exhibits a structure closely analogous to that of the original group $\chi(H)$ (cf. \cite[Subsection 4.4]{sidki-wp}), as we will see in the next theorem; from now on, according to  \cite[Section 4]{sidki-wp} we denote by $D$ the subgroup $D(H) = [H, H^\psi] \leqslant \chi(H)$.

\begin{thmB} \label{th:thmB}  
\textit{If $H$ is a perfect group, then}
\begin{itemize}
       \item[(i)]  $D_{ij} = D_{ks} \cong D, \ \forall\, 0 \leq i < j \leq n-1, \ 0 \leq k < s \leq n-1$; 
       \item[(ii)] $\chi(n,H)\cong ( \ldots ((D\cdot H) \cdot H^{\psi}) \cdots ) \cdot H^{\psi_{n-1}} \, \cong T_n(\tilde{H})$.
\end{itemize}
\end{thmB}
Here and elsewhere in the text, the dots stand for (internal) semi-direct products; the group $T_n(\tilde{H})$ above is a generalization of the group $T(\tilde{H})$ (cf. \cite[Subsection 2.3]{sidki-wp}) which is constructed in Subsection~\ref{ssec:tn}.

Concerning nilpotency, we establish the following bounds for the group $\chi(n,H)$.
\begin{thmC}  \label{th:thmC}
Let $H$ be an $m$-generated nilpotent group of class at most $c$ such that $m \geq 2$, $c\geq 1$, and let $G:=\chi(n,H)$.
\begin{itemize}
    \item[(i)] If $m = 2$, then $\gamma_{c+2}(G)=1$;
    \item[(ii)] If $2 < m \leqslant c+2$, then $\gamma_{c+3}(G)=1$;
    \item[(iii)] If  $m \geqslant c+3$, then  $\gamma_{c+3}(G)$ is an elementary abelian $2$-group of rank at most
\begin{equation*}
\sum_{k=c+3}^{m}\binom{m}{k}\binom{n}{2}.
\end{equation*}
\end{itemize}
\end{thmC}

Turning to the mother group $\chi_n(H)$, we obtain a dichotomy for $n = 3$ that contrasts sharply with the finiteness behaviour of $\chi(n,H)$.

\begin{thmD} \label{th:thmD}
\textit{Let $H$ be a $2$-generated abelian $p$-group.}
\begin{itemize}
    \item[(i)] \textit{If $p$ is odd and $H \cong C_{p^t} \times C_{p^u}$ with $t \geq u \geq 1$, then $\chi_3(H)$ is finite of order dividing $p^{3t+8u}$ and has nilpotency class at most $3$.}
    \item[(ii)] \textit{If $H = C_2 \times C_2$ (the Klein four-group), then $\chi_3(H)$ is infinite, and the kernel of the natural epimorphism $\pi : \chi_3(H) \twoheadrightarrow \chi(3,H)$ has $\mathbb{Z}^4$ as a homomorphic image.}
\end{itemize}
\end{thmD}

The paper is organized as follows. In Section~\ref{sec:prelim}, we discuss some preliminary results; in particular, in Subsection~\ref{ssec:embed} we establish the embedding result $\chi_m(H) \to \chi_n(H)$ for all $m \leq n$, show that the corresponding embedding $\chi(H) \to \chi(n,H)$ holds when $H$ is abelian, perfect, or satisfies $R(H)=1$, and exhibit counterexamples showing that it fails in general (Remark~\ref{rem:embed_fails}), and discuss some properties of the iterated groups $\chi^n(H)$; in Subsection~\ref{ssec:diag} we recall the diagonal representations $\delta_n:\chi_n(H)\to\mathbf{X}_{2^n-1}(H)$, which come from the Hamming codes, and, in Subsection~\ref{ssec:tn}, we construct the group $T_n(\tilde{H})$, an $n$-fold generalization of the group $T(\tilde{H})$ of \cite{sidki-wp}, and establish its basic properties. In Section~\ref{sec:chin}, we discuss the group $\chi(n,H)$ and prove Theorems~\hyperref[th:thmA]{A} and~\hyperref[th:thmB]{B}, as well as Theorem~\hyperref[th:thmC]{C} for this group.
In Section~\ref{sec:chi_n_p} we study the mother group $\chi_3(H)$ for $H$ abelian and prove Theorem~\hyperref[th:thmD]{D}; this dichotomy also shows the necessity of the relations $[D_{ij}, L_{kl}]=1$ in the definition of $\chi(n,H)$ to ensure finiteness.  In Section~\ref{sec:hybrid} we compare the iterated group $\chi^2(H)$ with the mother group $\chi_4(H)$, both generated by four copies of $H$, and determine their largest common quotient $Q(H)$; we close with diagonal representations whose images are co-nilpotent in the direct power, and with a question on their growth with respect to $n$. In Appendix~\ref{app:hamming}, D.~Savchuk constructs the diagonal representation of $\chi_n(H)$ in $\mathbf{X}_{2^n-1}(H)$ by means of the parity check matrix of the Hamming error-correcting code, and proves that its image contains $\mathbf{X}_{2^n-1}(\gamma_n(H))$.

\section{Preliminary Results} \label{sec:prelim}
In this section we establish some preliminary results that will be used in the sequel. Let us start by discussing the embedding issue.
\subsection{Embeddability} \label{ssec:embed}
Let us write $G(n)$ as the group representing our classes $\chi_n(H)$ and $\chi(n,H)$ with their corresponding presentations written simply as 
\begin{equation*}
  \langle H_0, \ldots, H_{n-1}  \mid  S \, \rangle,  
\end{equation*}
where $H_i$ stands for $H^{\psi_i}, \; i=0, \ldots, n-1$, and $S$ stands for the corresponding sets of relations, according to Definitions~\ref{def:mother_grp} and~\ref{def:chin_1}.
Thus we have 
\begin{equation} \label{eq:G_n}
G(n) = \langle H_{(i, n)} \mid i = 0, \ldots, n-1 \rangle, 
\end{equation}
where $H_{(i, n)}$ is the image of $H_i$ in $G(n)$, so that $G(2)$ is the original group $\chi(H).$ With this notation we can state the
\begin{proposition} \label{prop:embed_mother}
   There is an embedding of $\chi_{n-1}(H)$ into $\chi_n(H)$ for all $n \geqslant 3.$
\end{proposition}
\begin{proof}
Writing $K_n$ for the subgroup $\langle \, H_{(i,n)} \mid i=0, \ldots, n-2 \, \rangle \leqslant \chi_n(H)$, and, for $h \in H$, $h_{(i,n)}$ for the image of $h^{\psi_i}$ in $\chi_n(H)$, we get an epimorphism
\begin{equation*}
\lambda: \chi_{n-1}(H) = \langle H_{(i, n-1)} \mid i = 0, \ldots n-2 \rangle \to K_n, \; h_{(i, n-1)} \mapsto h_{(i,n)}.
\end{equation*}
On the other hand, the assignment $h_{(i,n)} \mapsto h_{(i, n-1)}$ for $0 \leqslant i \leqslant n-2$ and $h_{(n-1,n)}  \mapsto e$ preserves the defining relations of $\chi_n(H)$ --- the image of each weak commutativity relator $[h^{\psi_i}, h^{\psi_j}]$ is again such a relator, or trivial --- and hence extends to an epimorphism $\mu: \chi_n(H) \to \chi_{n-1}(H)$. Let $\mu'$ be the restriction of $\mu$ to the subgroup $K_n$. Then $\lambda \mu' = \Id$, so $\lambda$ is an isomorphism, giving the desired embedding of $\chi_{n-1}(H)$ into $\chi_n(H).$
\end{proof}

For the class $\chi(n,H)$ the retraction above is \emph{not} available in general: $\mu$ carries the relators $[D_{ij}, L_{k,n-1}]=1$ of $\chi(n,H)$ to relations of the form $[D_{ij}, H^{\psi_k}]=1$, which need not hold in $\chi(n-1,H)$. In fact, the embedding itself may fail (Remark~\ref{rem:embed_fails} below). The next lemma reduces the question to the behaviour of the subgroup $D$.

\begin{lemma}\label{lem:embred}
For $n \geqslant 3$ let $\lambda: \chi(H) \to \langle H, H^{\psi}\rangle \leqslant \chi(n,H)$ be the epimorphism defined by $h \mapsto h$, $h^{\psi}\mapsto h^{\psi}$. Then $\ker\lambda \subseteq D$. Consequently, $\lambda$ is an isomorphism if and only if $\lambda|_{D}: D \to D_{01}$ is injective; for $H$ finite, if and only if $|D_{01}| = |D|$. The same holds for the pair $(i,j)$ in place of $(0,1)$, by the symmetry of the defining presentation under permutations of the copies.
\end{lemma}
\begin{proof}
Composing $\lambda$ with the quotient map $\chi(n,H) \twoheadrightarrow \chi(n,H)/C \cong \mathbf{X}_n(H)$ of Proposition~\ref{prop:epichi}, followed by the projection of $\mathbf{X}_n(H)$ onto the coordinates $0$ and $1$, yields the homomorphism $\alpha: \chi(H) \to H \times H$, $h \mapsto (h,1)$, $h^{\psi} \mapsto (1,h)$, whose kernel is $D$ \cite[Section~4.1]{sidki-wp}. Hence $\ker\lambda \subseteq \ker\alpha = D$, and $\ker\lambda = \ker(\lambda|_D)$.
\end{proof}

\begin{proposition} \label{prop:embed}
Let $n \geqslant 3$. Then $\lambda$ is an isomorphism --- that is, $\chi(H)$ embeds into $\chi(n,H)$ --- in each of the following cases:
\begin{itemize}
\item[(i)] $R(H)=1$;
\item[(ii)] $H$ abelian;
\item[(iii)] $H$ perfect.
\end{itemize}
\end{proposition}
\begin{proof}
(i). Recall that $R(H)=[H,L,H^{\psi}]$. Since $[L,H]=D$ \cite[Section~4.1]{sidki-wp}, we have $R(H)=[D,H^{\psi}]$; applying the swap automorphism of $\chi(H)$ (which interchanges $H$ and $H^{\psi}$ and fixes $D$) we get $[D,H] \cong [D,H^{\psi}]$, so the hypothesis forces $[D,H]=[D,H^{\psi}]=1$. Now the assignment $h \mapsto h$, $h^{\psi}\mapsto h^{\psi}$, $h^{\psi_k}\mapsto e$ ($2 \leq k \leq n-1$) sends every weak commutativity relator of $\chi(n,H)$ to such a relator or to the identity, and each relator $[[x^{\psi_i},y^{\psi_j}],\, z^{-\psi_k}z^{\psi_l}]$ to a commutator of one of the forms $[d,z^{\pm1}]$, $[d,z^{\pm\psi}]$, $[d,l]$ (with $d \in D$, $l \in L$), or to the identity; all of these are trivial in $\chi(H)$, by the first paragraph and by $[D,L]=1$ \cite[Proposition~4.1.7]{sidki-wp}. Hence it extends to an epimorphism $\mu:\chi(n,H)\to\chi(H)$ with $\mu\lambda=\Id_{\chi(H)}$, and $\lambda$ is injective.

(ii). For $H$ abelian consider instead the \emph{folding} $h \mapsto h$, $h^{\psi}\mapsto h^{\psi}$, $h^{\psi_k}\mapsto h$ ($2 \leq k \leq n-1$). Weak commutativity relators are sent to $[h,h]$, $[h,h^{\psi}]$ or $[h^{\psi},h]$, all trivial. As for the mixed relators, every copy $H^{\psi_i}$ with $i\neq1$ is sent to $H$, so each $D_{ij}$ with $1\notin\{i,j\}$ is sent to $[H,H]=1$ ($H$ abelian), each $D_{1j}$ is sent to $[H^{\psi},H]=D$, and each $L_{kl}$ is sent into $L^{\pm1}\cup\{e\}$; the images of the mixed relators are therefore trivial, by $[D,L]=1$. This again yields a retraction $\mu'\lambda=\Id$.

(iii). By Lemma~\ref{lem:embred} it suffices to show that $\lambda|_D$ is injective. The epimorphism $\lambda_{\tilde{H}}$ of Proposition~\ref{prop:epichintn} carries $D_{01}$ onto $\overline{D}_{01}$, and for $H$ perfect one has $\overline{D}_{01}\cong D$ (Remark~\ref{rem:epialfa} and \cite[Section~4.4]{sidki-wp}; see the proof of Theorem~\hyperref[th:thmB]{B}). Thus the composite $D \to D_{01} \to \overline{D}_{01}$ is a surjection between groups of the same finite order, forcing $\lambda|_D$ to be injective (for infinite perfect $H$ the same argument applies, as the composite is an isomorphism).
\end{proof}

\begin{remark}\label{rem:embed_fails}
The hypotheses in Proposition~\ref{prop:embed} cannot be dropped: \emph{the embedding fails in general}. Computations with GAP \cite{gap} show that for $H=(C_4\times C_4)\rtimes C_2$ (\texttt{SmallGroup(32,34)}) the subgroup $\langle H,H^{\psi}\rangle$ of $\chi(3,H)$ has order $2^{17}$, whereas $|\chi(H)|=2^{18}$; thus $\ker\lambda\cong C_2$ and $D_{01}$ is a proper quotient of $D$. Similarly, for $H=C_3\times E$, where $E$ is the extraspecial group of order $27$ and exponent $3$ (\texttt{SmallGroup(81,12)}), one finds $\ker\lambda\cong C_3$. Two features are noteworthy: the phenomenon occurs for odd primes as well; and, since $\chi(E)$ \emph{does} embed into $\chi(n,E)$, the embedding property is not inherited by direct products. In particular, in contrast with the mother group $\chi_n(H)$, the relations $[D_{ij},L_{kl}]=1$ may impose new relations upon the subgroup $\langle H,H^{\psi}\rangle$.
\end{remark}

\begin{remark}\label{rem:S3_noepi}
Even when the embedding holds, no retraction --- indeed, no epimorphism at all --- from $\chi(n,H)$ onto $\chi(H)$ need exist: for $H=S_3$, GAP shows that $|\chi(3,S_3)|=648$, that $\langle H,H^{\psi}\rangle\cong\chi(S_3)$ (of order $108$), and that $\chi(3,S_3)$ has no normal subgroup of index $108$. Incidentally, in $\chi(3,S_3)$ the three subgroups $D_{01},D_{02},D_{12}$ coincide, so that the bound of Theorem~\hyperref[th:thmA]{A} is not attained in this case.
\end{remark}

    Since $G(2) = \chi(H)$ and the defining relations of $\chi(H)$ hold in $\langle H^{\psi_i}, H^{\psi_j}\rangle \leqslant G(n)$, we immediately obtain the
    \begin{corollary} \label{cor:embed}
  For every group $H$, all $n \geqslant 2$ and each pair $0\leq i<j\leq n-1$, there is an epimorphism $\lambda_{ij}:\chi(H)\twoheadrightarrow\langle H^{\psi_i},H^{\psi_j}\rangle\leqslant G(n)$, $h \mapsto h^{\psi_i}$, $h^{\psi} \mapsto h^{\psi_j}$, which carries $D$ onto $D_{ij}$. In particular, for $H$ finite, $|D_{ij}|$ divides $|D|$, and $D_{ij}$ inherits solvability, nilpotency and every law holding in $D$. Moreover, $\lambda_{ij}$ is an isomorphism whenever $G(n)=\chi_n(H)$ (Proposition~\ref{prop:embed_mother}), and whenever $G(n)=\chi(n,H)$ and $H$ is abelian, perfect, or satisfies $R(H)=1$ (Proposition~\ref{prop:embed}).
    \end{corollary}

We close this subsection by recording the behaviour of the operator $\chi$ under iteration. Recall from \eqref{def:iter_grp} that $\chi^{0}(H)=H$ and $\chi^{n}(H)=\chi\big(\chi^{n-1}(H)\big)$ for $n\geq 1$; we call $\chi^{n}(H)$ the \emph{$n$-th iterated weak commutativity group} of $H$. Since each application of $\chi$ doubles the number of copies of the argument, $\chi^{n}(H)$ contains $2^{n}$ isomorphic copies of $H$.

\begin{proposition}\label{prop:iter-embed}
For every group $K$ the natural map $\iota_K: K\to\chi(K)$, $k\mapsto k$, is a monomorphism. Consequently, $\chi^{n-1}(H)\hookrightarrow\chi^{n}(H)$ for every group $H$ and every $n\geq 1$.
\end{proposition}
\begin{proof}
The assignment $k\mapsto k$, $k^{\psi}\mapsto 1$ preserves the defining relations of $\chi(K)$, since $[k,k^{\psi}]\mapsto[k,1]=1$; hence it extends to a homomorphism $\rho_K:\chi(K)\to K$. As $\rho_K\,\iota_K=\Id_K$, the map $\iota_K$ is injective. Taking $K=\chi^{n-1}(H)$ yields $\chi^{n-1}(H)\hookrightarrow\chi\big(\chi^{n-1}(H)\big)=\chi^{n}(H)$.
\end{proof}

The following properties are preserved by $\chi$, and hence inherited by every iterate $\chi^{n}(H)$:
\begin{itemize}
\item[(i)] finiteness \cite[Corollary 3.6]{sidki-wp};
\item[(ii)] being a finite $p$-group \cite[Corollary 4.3.8]{sidki-wp};
\item[(iii)] solvability \cite[Corollary 4.1.8]{sidki-wp};
\item[(iv)] being a finite nilpotent group \cite[Corollary 4.3.9]{sidki-wp};
\item[(v)] local nilpotency \cite[Lemma 2.5]{grs};
\item[(vi)] being polycyclic, or polycyclic-by-finite \cite{LO};
\item[(vii)] finite presentability \cite{BD}.
\end{itemize}

\subsection{Diagonal Representations} \label{ssec:diag}

The first access to the structure of $\chi_2(H) = \langle H, H^\psi \rangle$ in \cite{sidki-wp} was obtained through its \emph{diagonal representation}
\[
    \delta_2 : \chi_2(H) \longrightarrow \mathbf{X}_3(H),
\]
where $\mathbf{X}_k(H)$ denotes the direct power $H \times \cdots \times H$ of $k$ copies of $H$, and $\delta_2$ is defined by

\[
    h \longmapsto (h, e, h), \qquad h^\psi \longmapsto (e, h, h)
\]
for all $h \in H$, which is codable as the $2 \times 3$ matrix
\[
   M_2:= \begin{pmatrix} 1 & 0 & 1 \\ 0 & 1 & 1 \end{pmatrix}.
\]
Denote the image of $\chi_2(H)$ under $\delta_2$ by $T_2(H)$. Then it follows simply that the derived subgroups of $T_2(H)$ and $H$ satisfy the relation
\[
    T_2(H)' = \mathbf{X}_3(H').
\]
A generalization of the above formula is obtained in Appendix~\ref{app:hamming} for $\chi_n(H)$ ($n \geq 2$), where a diagonal representation
\[
    \delta_n : \chi_n(H) \longrightarrow \mathbf{X}_{2^n-1}(H)
\]
is defined using the Hamming error-correcting code given by an $\bigl(n \times (2^n-1)\bigr)$ matrix over $\mathbb{Z}_2$.  For instance, for $n=3$ the representation $\delta_3$ is codable as the $3 \times 7$ matrix 
\[ M_3 := \begin{pmatrix} 1 & 0 & 1 & 0 & 1 & 0 & 1 \\ 0 & 1 & 1 & 0 & 0 & 1 & 1 \\
0 & 0 & 0 & 1 & 1 & 1 & 1
\end{pmatrix}.
\] 
This produces the relation
\[
    \gamma_n\bigl(T_n(H)\bigr) = \mathbf{X}_{2^n-1}\bigl(\gamma_n(H)\bigr),
\]
relating the $n$-th terms of the lower central series of $T_n(H)$ and of $H$, where $T_n(H)$ is the image of $\chi_n(H)$ under $\delta_n$.

\subsection{The Group $T_n(\tilde{H})$} \label{ssec:tn}

We define an $n$-version of the group $T(\tilde{H})$ introduced in \cite[Subsection 2.3]{sidki-wp}. To this end, let $H$ be any group and let $\tilde{H}$ be a stem extension of $H$, in the sense that there exists a subgroup $Z \leq \tilde{H}^{\prime } \cap Z(\tilde{H})$ such that $\tilde{H}/Z \cong H$.

Set $\tilde{T_n} = \tilde{H}^{n+1}$ and $(a)_i=(1,\ldots,1, a, 1,\ldots,1) \in 
\tilde{T_n}$, where $a \in \tilde{H}$ occurs at the $i$-th coordinate, $1 \leq i \leq n + 1$. 

Now we define:   
\[
\phi=(2,3, \ldots, n+1) \in S_{n+1}; 
\]
\[
\tilde{\varphi} \in \text{Aut}(\tilde{T}_n) \; \text{such that} \; \tilde{\varphi}: (a)_i \mapsto (a)_{\phi(i)}, \,1 \leq i \leq n+1; 
\]
\[
\tilde{H_0} = \left< (a,a,1,1,\ldots,1) \mid \; \; a \in \tilde{H}
\;\right>, \quad \tilde{H_i}={\tilde{H_0}}^{\tilde{\varphi}^{i}}, \; 1 \leq i
\leq n-1; 
\]
\[
Z_0=\langle (z,z,1,1,\ldots,1) \mid \; \;z \in Z \;\rangle, \quad Z_i={Z_0}%
^{{\tilde{\varphi}}^{i}}, \; 1 \leq i \leq n-1;
\]
\[
\tilde{G} = \left< \tilde{H}_0, \; \tilde{H}_1, \ldots , \tilde{H}_{n-1}
\right>;
\]
\[
T_n(\tilde{H})=\dfrac{\tilde{G}}{Z_0Z_1\cdots Z_{n-1}};
\]
\begin{gather*}
H_0=\frac{\tilde{H_0}Z_0Z_1\cdots Z_{n-1}}{Z_0Z_1 \cdots Z_{n-1}}, \quad
H_1=\frac{\tilde{H_1}Z_0Z_1\cdots Z_{n-1}}{Z_0Z_1 \cdots Z_{n-1}}, \quad
\ldots, \\
H_{n-1}=\frac{\tilde{H}_{n-1}Z_0Z_1\cdots Z_{n-1}}{Z_0Z_1 \cdots Z_{n-1}}.
\end{gather*}

Thus we have $T_n(\tilde{H})= \langle H_0,\, H_1, \ldots, \; H_{n-1} \rangle$ and  $H
\cong H_0 \cong H_1 \cong \ldots \cong H_{n-1}$, as $H_i=\frac{\tilde{H_i}%
Z_0 Z_1\cdots Z_{n-1}}{Z_0 Z_1 \cdots Z_{n-1}} \cong \frac{\tilde{H_i}}{\tilde{H_i} \cap Z_0 Z_1
\cdots Z_{n-1}} = \frac{\tilde{H_i}}{Z_i}\cong H$ for all  $0 \leq i \leq n-1$.  Define  $\tilde{D}_{ij}=[%
\tilde{H_i},\tilde{H_j}] \quad \text{and} \quad \tilde{L}_{ij}=[\tilde{H_i}, 
\tilde{\varphi}^{j-i}]$, subgroups of $\tilde{G}$, and let $\overline{D}_{ij}$, $%
\overline{L}_{ij}$  be their respective images in $G\;(=T_n(\tilde{H}))$, for all $%
0\leq i <j \leq n-1$.  It is easy to see that $\tilde{G}$ and $Z_0Z_1\cdots Z_{n-1}$ are $\tilde{\varphi}$-invariant and that $\tilde{\varphi}$ induces an automorphism $\varphi$ on $T_n(\tilde{H})$ such that
\begin{equation*}
\varphi: Z_0Z_1 \cdots Z_{n-1}(a)_i \mapsto Z_0Z_1 \cdots Z_{n-1}(a)_{\phi(i)} \quad \text{for all} \quad a \in \tilde{H}.
\end{equation*}
 Therefore,   
\begin{equation*}
H_i=H_0^{\varphi^{i}}, \; i=0, \ldots, n-1,  \quad \text{and} \quad T_n(\tilde{H})=\left<H_0,
H_0^{\varphi}, \ldots, H_0^{\varphi^{n-1}}\right>.
\end{equation*}

\begin{remark} \label{rem:epialfa} With the above settings, the following assertions hold:
\hskip -10pt 
\begin{itemize}
\item[(i)] For all $0 \leq i< j \leq n-1$, it follows that 
\begin{equation*}
\tilde{D}_{ij}, \; \tilde{L}_{ij} \, \triangleleft \tilde{T_n};
\end{equation*}
\begin{equation*}
\tilde{D}_{ij}= \left< ([a,b],1, \ldots,1) , \; a, \, b \in \tilde{H}
\right> =\tilde{H^{\prime }}\times 1\times \ldots \times1;
\end{equation*}
\begin{equation*}
\tilde{L}_{ij}=\left<(a^{-1})_i (a)_j \; \mid \; a \in \tilde{H} \right>;
\end{equation*}
\begin{equation*}
\tilde{D}_{ij} \cap Z_0Z_1 \ldots Z_{n-1} =1, \; \text{so that} \; \tilde{D}_{ij}
\cong \overline{D}_{ij}.
\end{equation*}
\item[(ii)] If $H$ is a perfect group, then  
\begin{equation*}
\tilde{D}_{ij}= \tilde{H}\times 1 \times \ldots \times 1. 
\end{equation*}%
\item[(iii)] For $\bar{x} = Z_0Z_1\cdots Z_{n-1}(x_0,x_1,\ldots,x_{n-1}) \in T_n({\tilde{H}})$ define
\begin{equation*}
\alpha_n: T_n(\tilde{H}) \rightarrow H^n, \bar{x} \mapsto (\bar{x}_0,\bar{x}_1, \ldots, \bar{x}_{n-1}).
\end{equation*}
Then $\alpha_n$ is an epimorphism such that
\begin{equation*}
\ker(\alpha_n)=\overline{D}_{ij}.
\end{equation*}
\end{itemize}
\end{remark}

\section{The Group \protect{$\chi(n,H)$}} \label{sec:chin}

We begin by introducing some preliminary considerations on the group $\chi(n,H)$.  
In particular, we prove basic structural properties that will be useful in the sequel.  
\begin{proposition}
\label{prop:epichi} There exists an epimorphism $\delta: \chi(n,H)
 \rightarrow H^n$ such that $\displaystyle \ker(\delta)=\prod_{0\leq i<j\leq n-1}{D_{ij}}$.
\end{proposition}

\begin{proof}
It is clear that the application  
\begin{eqnarray*}
&\chi(n,H) \rightarrow H^n & \\
& h \mapsto (h,1,\ldots,1) & \\
& h^{\psi} \mapsto (1,h, \ldots,1) & \\
& \vdots & \\
& h^{\psi_{n-1}} \mapsto (1,1,\ldots,h) &
\end{eqnarray*}
extends to an epimorphism $\delta: \chi(n,H) \rightarrow H^n$, as it
preserves the defining relations of $\chi(n,H)$. 
The following equalities hold: 
\begin{eqnarray*}
[h_1^{{\psi}_i},h_2^{{\psi}_j}]^{h_3^{{\psi}_k}}&=&[h_1^{{\psi}%
_i},h_2^{{\psi}_j}]^{h_3^{{\psi}_i}} \qquad (\text{since} \,[D_{ij},
L_{ki}]=1) \\
&=&[(h_1h_3)^{{\psi}_i},h_2^{{\psi}_j}][h_3^{{\psi}_i},h_2^{{\psi}_j}]^{-1}.
\end{eqnarray*}
So, $D_{ij} \triangleleft \chi(n,H)$ for all $i, j$ with $0 \leq i< j \leq n-1$, whence the normal subgroup $C :=\langle D_{ij}; 0\leq i < j \leq n-1 \rangle$ can be expressed as a product,
$$C=\displaystyle \prod_{0 \leq i <j \leq n-1} \!\! \!\! {D_{ij}}.$$  
Since $C \subseteq \ker(\delta)$, it follows that $\delta$ induces a homomorphism $\bar{\delta}$ given by  
\begin{eqnarray*}
&\bar{\delta}: \frac{\chi(n,H)}{C} \rightarrow H^n & \\
& Ch \mapsto (h,1,\ldots,1) & \\
& Ch^{\psi} \mapsto (1,h, \ldots,1) & \\
& \vdots & \\
& Ch^{\psi_{n-1}} \mapsto (1,1,\ldots,h) &
\end{eqnarray*}
On the other hand the application  
\begin{eqnarray*}
&H^n \rightarrow \frac{\chi(n,H)}{C}& \\
&(h,1, \ldots, 1) \mapsto Ch& \\
&(1,h, \ldots, 1) \mapsto Ch^{\psi}& \\
& \vdots & \\
&(1,1, \ldots, h) \mapsto Ch^{\psi_{n-1}},& \\
\end{eqnarray*}
preserves the defining relations of the group $H^n$, so that it can be
extended to the epimorphism $\alpha: H^n \rightarrow \frac{\chi(n,H)}{C} $.
As the composite map of these two last epimorphisms is the identity isomorphism, $\alpha \bar{%
\delta}= \Id$, the result follows. 
\end{proof}

As an immediate consequence of Corollary~\ref{cor:embed} we have the 
\begin{corollary} \label{cor:chin_sol}  
If $H$ is a solvable group, then so is $\chi(n,H)$.
\end{corollary}
\begin{proof}
 From \cite[Corollary 4.1.8]{sidki-wp} we know that $\chi(H)$ is solvable, since $H$ is solvable, and so is its subgroup $D = [H, H^\psi]$. Consequently, due to Corollary~\ref{cor:embed} each subgroup $D_{ij} \subseteq \chi(n,H)$, being an epimorphic image of $D$, is solvable for all pairs $(i, j), \, 0 \leq i < j \leq n-1$, whence  
$C = \prod_{0\leq i<j\leq n-1}{D_{ij}}$ is also solvable. Therefore, by Proposition~\ref{prop:epichi} we get that $\chi(n,H)$ is a solvable group. 
\end{proof}

\subsection{Bounding the order of the group \texorpdfstring{$\chi(n,H)$}{chi(n,H)}} 

As mentioned in the Introduction, in \cite{oliveira-thesis} it was obtained an upper bound for the order of the group $\chi(n,H)$ when $H$ is a finite abelian group. We are now ready to prove Theorem~\hyperref[th:thmA]{A}, which extends this result to the more general case where $H$ is any finite group. 

\begin{theorem} \label{th:upperbound}  If $H$ is a finite group, then $|\chi(n,H)|$ divides 
$|\chi(H)|^{\binom{n}{2}} |H|^{n-2\binom{n}{2}}$.
\end{theorem}
\begin{proof}
By Corollary~\ref{cor:embed}, $|D_{ij}|$ divides $|D|$ for all $\{i, j\}$ with $0 \leq i < j \leq n-1$; thus $|C|$ divides $|D|^{\binom{n}{2}}$, since
$\displaystyle C = \prod_{0 \le i < j \le n-1} D_{ij}$.  Moreover, $|\chi(n,H)| = |C| |H|^n$ thanks to Proposition~\ref{prop:epichi}. Now, since $|\chi(H)|=|D||H|^2$, we conclude that 
$|\chi(n,H)|$ divides $|D|^{\binom{n}{2}} \cdot |H|^n = |\chi(H)|^{\binom{n}{2}} \cdot |H|^{n-2\binom{n}{2}}$. 
\end{proof}

When $H$ is a finite abelian group of odd order, the subgroup $R(H) = [H, L, H^\psi]$ of $\chi(H)$ is trivial; this follows from the fact that in this case $R(H)^2 = 1$ (see \cite[Corollaries~4.1.15 and~4.2.3]{sidki-wp}).  
The following proposition shows that, in this case, the upper bound established in the preceding theorem is attained.  

\begin{proposition}
\label{prop:prodD}  If $H$ is a finite abelian group such that $R(H)=1$,
then $$\displaystyle \chi(n,H) \cong ( \dots ( ({\ \oplus_{0 \leq i <j \leq n-1}}
D_{ij}) \cdot H) \cdot H^{\psi} \cdots ) \cdot H^{\psi_{n-1}}.$$
\end{proposition}
\begin{proof}
We first consider the case $n=3$. By Proposition~\ref{prop:embed}, (i)--(ii),  we have an embedding $\phi: \chi(H) \hookrightarrow \chi(3,H)$, and
$D_{ij} \cong D$ for all $0 \leq i <j \leq 2$ (Corollary~\ref{cor:embed}). Now, consider the map $\phi_3$ defined by
\begin{equation}
     \label{hom:phi3}
\phi_3: \chi(3,H) \rightarrow \chi(H), \;  
 h \mapsto h, \;  
 h^{\psi}\mapsto h^{\psi} \; \text{and} \;  
h^{\psi_2} \mapsto 1, \; \text{for all} \; h \in H. 
\end{equation}
Note first that $\phi_3$ is well defined: since $R(H)=1$ we have $[D,H]=[D,H^{\psi}]=1$ in $\chi(H)$ (as in the proof of Proposition~\ref{prop:embed}(i)), so the images of the defining relations of $\chi(3,H)$ hold in $\chi(H)$. It is readily seen that $\phi_3$ is an epimorphism whose kernel is $\ker \phi_3 = D_{02}D_{12}H^{\psi_2}$. We will show that $D_{02}D_{12}H^{\psi_2}$ is in fact normal in $\chi(3,H)$ --- a property specific to this quotient, where the relations $[D_{ij}, L_{kl}] = 1$ hold for $\{i,j\}\neq\{k,l\}$ (these fail in the mother group $\chi_3(H)$).

\textit{Step 1: $[D_{02}, H] = [D_{02}, H^{\psi_2}] = 1$.}
Since $R(H)=1$, we have $[D,H]=[D,H^{\psi}]=1$ in $\chi(H)$, as in the proof of Proposition~\ref{prop:embed}(i); hence $D$ is central in $\chi(H)$, and applying the epimorphism $\lambda_{02}$ of Corollary~\ref{cor:embed} we get that $D_{02}$ is central in $\langle H, H^{\psi_2}\rangle$; in particular it commutes with $H$ and $H^{\psi_2}$.

\textit{Step 2: $[D_{02}, H^{\psi}] = 1$.}
This requires combining two facts. The relation $[D_{02}, L_{01}] = 1$ asserts that elements of $D_{02}$ commute with elements of $L_{01}$; since $a^{-1}\ap \in L_{01}$, for $d \in D_{02}$ and $a \in H$ we have $d \cdot a^{-1}\ap = a^{-1}\ap \cdot d$. Rearranging using $[d,a]=1$ from Step~1 yields $d^{\ap} = d$. Analogously $d^{\bp} = d$, so $[D_{02}, H^{\psi}] = 1$. (Note: the relation $[D_{02}, L_{01}] = 1$ alone gives only $d^{\ap} = d^{a}$, i.e., agreement of two conjugation actions; it is the centrality of $D_{02}$ in $\langle H, H^{\psi_2}\rangle$ that makes the right-hand side trivial.)

By Steps~1--2, $D_{02} \subseteq Z(\chi(3,H))$. The same reasoning, exchanging the roles of $H$ and $H^{\psi}$, gives $D_{12} \subseteq Z(\chi(3,H))$.

\textit{Step 3: $H^{\psi_2}$ is normalized modulo $D_{02}D_{12}$.}
For $h \in H$ and $g^{\psi_2} \in H^{\psi_2}$,
\[
h\, g^{\psi_2}\, h^{-1} = g^{\psi_2}\,[g^{\psi_2}, h^{-1}] \in H^{\psi_2} D_{02},
\]
and analogously $h^{\psi}\, H^{\psi_2}\, h^{-\psi} \subseteq H^{\psi_2}\, D_{12}$.
The three  steps above ensure the equality $\ker \phi_3 = D_{02}D_{12}H^{\psi_2}$ and, consequently, 
\begin{equation*}
|\chi(3,H)|=|\chi(H)||D_{02}D_{12}H^{\psi_2}|.
\end{equation*}
On the other hand, by Proposition~\ref{prop:epichi} there exists an
epimorphism $\delta: \chi(3,H) \rightarrow H^3$ whose kernel is $\ker \delta = D_{01} D_{02} D_{12}$, and so,   
\begin{equation*}
|\chi(3,H )|=|D_{01} D_{02} D_{12}||H|^3.
\end{equation*}
By comparing the orders we get   
\begin{eqnarray*}
|D_{01}D_{02}D_{12}||H|^3&=& |\chi(H)||D_{02}D_{12}H^{\psi_2}| \\
&=& |D||H|^2 |D_{02}D_{12}H^{\psi_2}|.
\end{eqnarray*}
But $H^{\psi_2} \cap D_{02}D_{12}=1$, because $H^{\psi_2} \cap
\ker \delta =1$.  That is,  
\begin{eqnarray*}
|D_{01}D_{02}D_{12}|&=&|D||D_{02}D_{12}| \\
&=& |D_{01}||D_{02}D_{12}|
\end{eqnarray*}
Finally we conclude that $D_{01} \cap D_{02}D_{12}=1$. By permuting the
indices in the last equality we get that $D_{01}D_{02}D_{12}$ is a direct
product.

For the general case, $n \geq 3$, we observe that the following applications: 
\[
\begin{array}{rcl}
\phi_4: \chi(3, H) & \longrightarrow & \chi(n, H) \\
h & \longmapsto & h \\
h^{\psi} & \longmapsto & h^{\psi} \\
h^{\psi_2} & \longmapsto & h^{\psi_2}
\end{array}
\qquad \text{and} \qquad 
\begin{array}{rcl}
\phi_5: \chi(n, H) & \longrightarrow & \chi(3, H) \\
h & \longmapsto & h \\
h^{\psi} & \longmapsto & h^{\psi} \\
h^{\psi_2} & \longmapsto & h^{\psi_2} \\
h^{\psi_j} & \longmapsto & 1 \quad   (j > 2)
\end{array},
\]
can be extended to homomorphisms ($\phi_5$ is well defined because $R(H)=1$, exactly as for $\phi_3$ above), where the composite $\phi_5\phi_4$ is the identity of $\chi(3,H)$; so we have
the inclusions
\begin{equation*}
\chi(H) \hookrightarrow \chi(3,H) \hookrightarrow \chi(n,H).
\end{equation*}
To upgrade this to a direct product decomposition of $C$, we show that
\begin{equation}\label{eq:direct-sum-cond}
D_{ij} \cap \prod_{\{k,l\} \neq \{i,j\}} D_{kl} = 1 \qquad \text{for each pair } \{i,j\} \subseteq \{0,1,\ldots,n-1\}.
\end{equation}
For each such pair, consider the projection
\begin{equation*}
\pi_{ij}: \chi(n,H) \rightarrow \chi(H), \qquad h^{\psi_i} \mapsto h, \quad h^{\psi_j} \mapsto h^{\psi}, \quad h^{\psi_t} \mapsto 1 \;\;\text{for}\;\; t \notin \{i,j\}.
\end{equation*}
The map $\pi_{ij}$ is well defined because $R(H)=1$ (as for $\phi_3$ above), and the epimorphism $\lambda_{ij}$ of Corollary~\ref{cor:embed} --- an isomorphism onto its image under the present hypothesis --- identifies $\chi(H)$ with $\langle H^{\psi_i}, H^{\psi_j}\rangle$ and $D$ with $D_{ij}$; since $\pi_{ij}$ is a left inverse of it, $\pi_{ij}$ restricts to an isomorphism $D_{ij} \xrightarrow{\sim} D$. Moreover, for any pair $\{k,l\} \neq \{i,j\}$ at least one of the indices $k, l$ lies outside $\{i,j\}$, so $\pi_{ij}(H^{\psi_k}) = 1$ or $\pi_{ij}(H^{\psi_l}) = 1$, giving $\pi_{ij}(D_{kl}) = [\pi_{ij}(H^{\psi_k}), \pi_{ij}(H^{\psi_l})] = 1$. Hence
\begin{equation*}
\pi_{ij} \;(\!\prod_{\{k,l\}\neq\{i,j\}} \!\! D_{kl}\,) = 1,
\end{equation*}
and any $x \in D_{ij} \cap \prod_{\{k,l\}\neq\{i,j\}} D_{kl}$ satisfies $\pi_{ij}(x) = 1$, forcing $x = 1$ by injectivity of $\pi_{ij}|_{D_{ij}}$. This establishes \eqref{eq:direct-sum-cond}, and therefore $C = \prod_{0 \leq i < j \leq n-1} D_{ij} = \bigoplus_{0 \leq i < j \leq n-1} D_{ij}$.
\end{proof}
As observed before, when $H$ is an abelian group it follows that $R(H)^2 = 1$.  
In particular, if $H$ is an abelian $p$-group with $p$ odd, then the group $R(H)$ is trivial.  
We thus obtain the following corollary.  

\begin{corollary} \label{cor:abelian_odd}
Suppose that $H$ is a finite abelian $p$-group, $p$ an odd prime. Then 
\[
\chi(n,H) \;\cong\;
( \dots ( ( \bigoplus_{0 \leq i < j \leq n-1} D_{ij} ) 
  \cdot H ) \cdot H^{\psi} \cdots ) \cdot H^{\psi_{\,n-1}} .
\]
\end{corollary}

The upper bound established in Theorem~\ref{th:upperbound} is not attained when $H$ is an elementary abelian $2$-group, as shown in the following proposition.  

\begin{proposition}
Let $H$ be an elementary abelian $2$-group of rank $k$. Then the order
of $\chi(n,H)$ divides $2^{\binom{k}{2}\binom{n}{2}+ \sum_{p=3}^{k}{\binom{k}{p}(n-1)}+nk}$.
\end{proposition}
\begin{proof}
From the epimorphism $\delta: \chi(n,H) \rightarrow H^n$ defined in
Proposition~\ref{prop:epichi}, we obtain
\begin{equation*}
|\chi(n,H)|=|C||H|^n, \; \text{where} \; C = \prod D_{ij}, \; \text{as above}.
\end{equation*}
Thus, for all $h_1, \; h_2, \; h_3, \; h_4 \;\in H, \;0 \leq i <j \leq n-1, \;
0\leq k <l \leq n-1$, we get
\begin{eqnarray*}
[h_1^{\psi_i},h_2^{\psi_j}]^{[h_3^{\psi_k},h_4^{\psi_l}]}&=&[h_1^{%
\psi_i},h_2^{\psi_j}]^{[h_3,h_4]} \qquad ( \text{since} \; [D_{ij},L_{kl}]=1) \\
&=& [h_1^{\psi_i},h_2^{\psi_j}] \qquad\qquad\quad (\text{once} \;H \;\text{is abelian}),
\end{eqnarray*}
and 
\begin{eqnarray*}
1&=&[(h_1^{\psi_i})^2, h_2^{\psi_j}] \\
&=&[h_1^{\psi_i}, h_2^{\psi_j}]^{h_1^{\psi_i}}[h_1^{\psi_i}, h_2^{\psi_j}] \\
&=& [h_1^{\psi_i}, h_2^{\psi_j}]^{h_1^{\psi_j}}[h_1^{\psi_i}, h_2^{\psi_j}]
\\
&=& [h_1^{\psi_i}, h_2^{\psi_j}][h_1^{\psi_i}, h_2^{\psi_j}] \\
&=& [h_1^{\psi_i}, h_2^{\psi_j}]^2.
\end{eqnarray*}
Then $C$ is an elementary abelian $2$-group. Suppose that $H= \left< a_1,a_2, \,
\ldots \, , a_k \right>$; by commutator relations and by the relations $%
[D_{ij}, L_{ks}]=1$ for all $0\leq i< j \leq n-1,\; \; 0\leq k< s \leq n-1$, we conclude that the commutator $[h_1^{\psi_i},h_2^{\psi_j}] \in C$ is a product of
commutators of type $[a_r^{\psi_i},a_s^{\psi_j}]^{a_1^{\epsilon_1}a_2^{%
\epsilon_2}\cdots a_k^{\epsilon_k}}$, where $\epsilon_i \in \{0,1\}$. But $%
[a_r^{\psi_i}, a_s^{\psi_j}]^{a_1^{\epsilon_1}}=[a_r^{\psi_i},
a_s^{\psi_j}][a_r^{\psi_i}, a_s^{\psi_j},a_1^{\epsilon_1}]$. So, 
\begin{multline*}
[a_r^{\psi_i},a_s^{\psi_j}]^{a_1^{\epsilon_1}a_2^{\epsilon_2}\cdots a_k^{\epsilon_k}}
= [a_r^{\psi_i},a_s^{\psi_j}] \prod_{1\leq l \leq k}[a_r^{\psi_i},a_s^{\psi_j},a_l^{\epsilon_l}] \\
\cdot \prod_{1\leq l<t \leq k}[a_r^{\psi_i},a_s^{\psi_j},a_l^{\epsilon_l},a_t^{\epsilon_t}] \cdots
[a_r^{\psi_i},a_s^{\psi_j},a_1^{\epsilon_1}, \ldots, a_k^{\epsilon_k}].
\end{multline*}
Using Witt's identity on commutators of degree higher than 2 we get
\begin{eqnarray*}
[a_r^{\psi_i},a_s^{\psi_j},a_{i_1}^{\epsilon^{i_1}}, \ldots ,
a_{i_t}^{\epsilon^{i_t}}]
&=&[[a_s^{\psi_j},a_{i_1}^{\epsilon^{i_1}},a_r^{\psi_i}][a_{i_1}^{%
\epsilon^{i_1}},a_r^{\psi_i},a_s^{\psi_j}], \ldots ,
a_{i_t}^{\epsilon^{i_t}}] \\
&=&[a_s^{\psi_j},a_{i_1}^{\epsilon^{i_1}},a_r^{\psi_i}, \ldots ,
a_{i_t}^{\epsilon^{i_t}}][a_{i_1}^{\epsilon^{i_1}},a_r^{\psi_i},a_s^{%
\psi_j}, \ldots , a_{i_t}^{\epsilon^{i_t}}] \\
&=&[a_s^{\psi_j},a_{i_1}^{\epsilon^{i_1}},a_r, \ldots ,
a_{i_t}^{\epsilon^{i_t}}][a_r^{\psi_i},a_{i_1}^{\epsilon^{i_1}},a_s, \ldots
, a_{i_t}^{\epsilon^{i_t}}]
\end{eqnarray*}
and, using Witt's identity again, we obtain  
\begin{eqnarray*}
[a_r^{\psi_i}, \ldots,
a_{i_o}^{\epsilon^{i_o}},a_{i_{p}}^{\epsilon^{i_{p}}},a_{i_q}^{%
\epsilon^{i_q}}\ldots] &=&\biggl[[a_{i_p}^{\epsilon^{i_p}},a_{i_{q}}^{%
\epsilon^{i_{q}}},[a_r^{\psi_i}, \ldots,a_{i_o}^{\epsilon^{i_o}}]] \times \\
& \times &[a_{i_{q}}^{\epsilon^{i_{q}}},[a_r^{\psi_i},
\ldots,a_{i_o}^{\epsilon^{i_o}} ],a_{i_p}^{\epsilon^{i_p}}], \ldots,%
\biggr] \\
&=&\biggl[a_{i_{q}}^{\epsilon^{i_{q}}},[a_r^{\psi_i},
\ldots,a_{i_o}^{\epsilon^{i_o}} ],a_{i_p}^{\epsilon^{i_p}}, \ldots,\biggr] \\
&=&\biggl[a_r^{\psi_i}, \ldots,a_{i_o}^{\epsilon^{i_o}}
,a_{i_{q}}^{\epsilon^{i_{q}}},a_{i_p}^{\epsilon^{i_p}}, \ldots,\biggr]. \\
\end{eqnarray*}
That is, $C$ is generated by commutators of degree 2 of type $%
[a_r^{\psi_i},a_s^{\psi_j}]$, with $0\leq i< j\leq n-1$ and $1\leq r< s \leq k$%
, together with other commutators of higher degree, of type $[a_r^{\psi_i},
a_{i_1}, a_{i_2}, \ldots , a_{i_t}]$, where $1\leq r<i_1<i_2< \ldots i_t
\leq k$. Thus $C$ has at most $\binom{k}{2}\binom{n}{2}$ generators of degree
2, and at most $\binom{k}{p}(n-1)$ generators of degree higher than 2.
Therefore, the order of $C$ divides $2^{\binom{k}{2}\binom{n}{2}+
\sum_{p=3}^{k}{\binom{k}{p}(n-1)}}$ and thus, the order of $\chi(n,H)$
divides $2^{\binom{k}{2}\binom{n}{2}+ \sum_{p=3}^{k}{\binom{k}{p}(n-1)}+nk}$, proving our assertion.
\end{proof}
Notice that by \cite[Theorem 4.2.4]{sidki-wp}, if $H$ is an  elementary abelian $2$-group, then   
\begin{equation*}
|\chi(H)|=2^{2^k-1}2^k \qquad \text{and} \qquad |D|=\frac{2^{2^k-1}2^k}{2^{2k}}.
\end{equation*}
From the equality $2^k=(1+1)^k=\binom{k}{0}+\binom{k}{1}+ \ldots +
\binom{k}{k}$, we get that  
\begin{equation*}
|D|=2^{\binom{k}{2}+\binom{k}{3}+ \ldots + \binom{k}{k}}
\end{equation*}
This implies that $2^{\binom{k}{2}\binom{n}{2}+ \sum_{p=3}^{k}{\binom{k}{p}(n-1)}}$ is
less than $|D|^{\binom{n}{2}}=2^{\binom{k}{2}\binom{n}{2}+\binom{k}{3}\binom{n}{2}+
\ldots + \binom{k}{k}\binom{n}{2}}$.

\subsection{$\protect\chi(n,H): H$ a Perfect Group}

In \cite[Theorem 4.4.1]{sidki-wp}, it was proven that if $H$ is a perfect group, then $\chi(H)$ is isomorphic to the group $T(\tilde{H})$ (cf. \cite[Subsection 4.4]{sidki-wp}) for a suitable stem extension $\tilde{H}$ of $H$; in this situation we also have $D \cong \tilde{H}$. We now relate $\chi(n,H)$ to the group $T_n(\tilde{H})$ constructed in Subsection~\ref{ssec:tn}, thereby extending the corresponding result of \cite{sidki-wp} to $n$ copies of $H$.

Recall that by Proposition~\ref{prop:epichi} there exists an epimorphism $\delta: \chi(n,H)
 \rightarrow H^n$ such that $\displaystyle \ker(\delta)=\prod_{0\leq i<j\leq n-1}{D_{ij}} \quad (= C)$.
\begin{proposition}
\label{prop:epichintn} There exists an epimorphism $\lambda_{\tilde{H}}:
\chi(n,H) \rightarrow T_n(\tilde{H})$, such that $\ker(\lambda_{\tilde{H}})
\subseteq C$.
\end{proposition}
\begin{proof}
Define the application $\lambda : \chi(n,H)\rightarrow T_n(\tilde{H})$ on the generators of $\chi(n, H)$ putting 
\begin{equation*}
\begin{split}
\lambda(h^{\psi_i}) & = Z_0Z_1\cdots Z_{n-1}(h,h,1,\ldots,1)^{\varphi^i} \\
\ & = Z_0 Z_1\cdots Z_{n-1}{(h)_1(h)_{i+2}}, \quad i = 0, \ldots, n-1.
\end{split}
\end{equation*} 
So we have,
\begin{equation}  \label{eq:comdiah}
[Z_0Z_1\cdots Z_{n-1} (a)_1(a)_i, \, Z_0Z_1 \cdots Z_{n-1} (a)_1(a)_j]=Z_0Z_1\cdots Z_{n-1}
\end{equation}
for all $a \in \tilde{H}$. On the other hand,   
\begin{eqnarray*}
& [Z_0Z_1\cdots Z_{n-1}(a)_1(a)_i,Z_0Z_1\cdots
Z_{n-1}(b)_1(b)_j]=Z_0Z_1\cdots Z_{n-1}([a,b])_1,& \\
&\left(Z_0Z_1\cdots Z_{n-1}(a)_1(a)_k\right)^{-1}Z_0Z_1\cdots
Z_{n-1}(a)_1(a)_s=Z_0Z_1\cdots Z_{n-1}(a^{-1})_k(a)_s.&
\end{eqnarray*}
Consequently,  
\begin{equation}  \label{eq:comdelh}
[Z_0Z_1\cdots Z_{n-1}([a_1,a_2])_1, \; Z_0Z_1\cdots Z_{n-1}(a^{-1})_k(a)_s]=Z_0Z_1\cdots Z_{n-1}.
\end{equation}
Equations \eqref{eq:comdiah} and \eqref{eq:comdelh} show that the images by $\lambda$ of the generators $H^{\psi_0} = H, H^{\psi_1} = H^{\psi}, \ldots, H^{\psi_{n-1}}$ of the group $\chi(n, H)$ also generate $T_n(\tilde{H})$ and satisfy the defining relations of $\chi(n, H)$. Therefore, $\lambda$ extends to an epimorphism   
\begin{equation*}
\lambda_{\tilde{H}} : \chi(n,H) \rightarrow T_n(\tilde{H}),
\end{equation*}
such that  
\begin{equation*}
\lambda_{\tilde{H}}(H^{\psi_i})=H_0^{\varphi^i} \; \; \text{for all} \; \; 1\leq i \leq n-1,
\quad \text{and} \quad 
\lambda_{\tilde{H}}(D_{ij})=\overline{D}_{ij}. 
\end{equation*}
By Remark \ref{rem:epialfa} we have that $C \; \; (=
\langle D_{ij} \; ; \; 0\leq i < j \leq n-1 \rangle) \; = Ker(\lambda_{\tilde{H}%
}\alpha_n)$, that is, $\ker(\lambda_{\tilde{H}}) \subseteq C$. This proves our assertion.
\end{proof}
Recall that we are writing $D$ for the subgroup $D(H) = [H, H^\psi] \leq \chi(H).$ We are ready to establish Theorem~\hyperref[th:thmB]{B}, restated as follows.
\begin{theorem} \label{thm:perfect}
Let $H$ be a perfect group. Then,
\begin{itemize}
\item[($i$)] $D_{ij}=D_{ks}\cong D, \ \forall\, 0 \leq i < j \leq n-1,
\ 0 \leq k < s \leq n-1$; 
\item[($ii$)] $\chi(n,H) \cong \left( \ldots \left( \left( D \cdot H \right) \cdot H^{\psi}\right) \cdots \right) \cdot H^{\psi_{n-1}} \cong T_n(%
\tilde{H})$. 
\end{itemize}
\end{theorem}
\begin{proof} 
(i). By \cite[Theorem 4.4.1, Corollary 4.4.2 and Remark 2.3.1-(vi)]{sidki-wp}, if $H$ is a perfect group then there exists a maximal stem extension $\tilde{H}$ of $H$ such that $\chi(H) \cong T(\tilde{H})$; in addition, $D \cong  \tilde{H}$. Thus, by Remark~\ref{rem:epialfa},  
\begin{equation*}
\bar{D}_{ij} \cong \tilde{H} \cong D \; \; \text{for all} \; \; 0 \leq i<j \leq n-1.
\end{equation*}
Now, Proposition~\ref{prop:epichintn} and its proof show that $D \cong \bar{D}_{ij}$ is the image of $D_{ij}$ by the epimorphism 
$\lambda_{\tilde{H}}:
\chi(n,H) \rightarrow T_n(\tilde{H})$, for all  $ 0 \leq i<j \leq n-1$. On the other hand, $%
D_{ij}$ is the image of $D$ under the epimorphism $\lambda_{ij}$ of Corollary~\ref{cor:embed}, so that $|D_{ij}|$ divides $|D|$. Therefore, $D \cong D_{ij} \; \; \text{for all} \; \;0\leq i<j \leq n-1$.
By the relations $[D_{ij}, L_{ks}]=1$ holding in  $\chi(n,H)$, it follows that  
$[D_{ij}, D_{ij}]=[D_{ij}, D_{ks}]=[D_{ks}, D_{ks}], \; \; \text{for all} \; \;  \;0\leq i<j \leq n-1, 0\leq k<s \leq n-1.$
But $D \cong D_{ij}$ is also a perfect group, and therefore $D_{ij} = D_{ks}$ for all 
$0 \leq i < j \leq n-1, \quad 0\leq k<s \leq n-1$. This proves part (i). 

(ii). Note that the isomorphism $\chi(n,H) \cong \left( \ldots \left( \left( D \cdot H \right) \cdot H^{\psi}\right) \cdots \right) \cdot H^{\psi_{n-1}}$
is a consequence of the previous item and of   
Proposition~\ref{prop:epichi}. By Proposition~\ref{prop:epichintn} and Remark~\ref{rem:epialfa}
 we have that $\ker(\lambda_{\tilde{H}}) \subseteq D_{ij}$. Now, each of the subgroups $%
\langle H^{\psi_i}, H^{\psi_j} \rangle$ of $\chi(n,H)$ is isomorphic with $%
\chi(H)$: indeed, since $D \cong D_{ij}$ (finite), the restriction $\lambda_{ij}|_D$ is injective, and Lemma~\ref{lem:embred} then shows that $\lambda_{ij}$ is an isomorphism. From \cite[Theorem 4.4.1]{sidki-wp}  and by the 
isomorphism $\lambda_{\tilde{H}}|_ {\left<H^{\psi_i},H^{\psi_j} \right>}$,
we conclude that $D_{ij} \cap Ker(\lambda_{\tilde{H}})=1$, for all $0\leq i<j \leq n-1$, which finally
implies that $\ker(\lambda_{\tilde{H}})=1$. This concludes the proof.
\end{proof}

\subsection{$\protect\chi(n,H): H$ Nilpotent} \label{subsec:nil_chi}

We close this section by establishing the nilpotency results for $\chi(n,H)$, originally established in \cite{oliveira-thesis} and extending earlier results of \cite{grs} for $\chi(H)$. The arguments rely solely on weak commutativity together with the relation
\begin{equation}\label{eq:DLcomut}
[D_{ij}, L_{ks}]=1, \qquad \text{for all } \{i,j\}\neq\{k,s\}, \; 0\leq i<j,\, k<s \leq n-1;
\end{equation}
which yields, in particular, that the conjugation action of any copy $H^{\psi_k}$ on $D_{ij}$ coincides with that of $H^{\psi_i}$ (equivalently $H^{\psi_j}$): for $d\in D_{ij}$ and $z\in H$ one has $d^{z^{\psi_k}}=d^{z^{\psi_i}}=d^{z^{\psi_j}}$, for all $0\leq k\leq n-1$ (see Lemma~\ref{lem:relacoes_chi}(ii)). These two facts --- weak commutativity and this relabelling of the conjugating copy --- are the only ingredients of the proofs given below. Throughout this subsection we write $\chi_n := \chi(n,H)$.

A further structural feature, used tacitly throughout, is the symmetry among the copies.

\begin{lemma}\label{lem:sym_aut}
The symmetric group $S_n$ acts on $\chi(n,H)$ by automorphisms permuting the copies: for each $\sigma \in S_n$ the assignment $h^{\psi_i} \mapsto h^{\psi_{\sigma(i)}}$ $(h\in H, \; 0\leq i\leq n-1)$ extends to an automorphism $\alpha_\sigma$ of $\chi(n,H)$, and $\sigma \mapsto \alpha_\sigma$ is a homomorphism $S_n \to \operatorname{Aut}(\chi(n,H))$.
\end{lemma}
\begin{proof}
The defining relations in Definition~\ref{def:chin_1} are invariant under permuting the copies: weak commutativity $[h^{\psi_i},h^{\psi_j}]=1$ holds for every pair, and each relation $[D_{ij},L_{kl}]=1$ (with $\{i,j\}\neq\{k,l\}$) is carried by $\sigma$ to $[D_{\sigma(i)\sigma(j)},L_{\sigma(k)\sigma(l)}]=1$, a relation of the same type. Hence $\alpha_\sigma$ is a well-defined endomorphism; since $\alpha_{\sigma^{-1}}$ is a two-sided inverse, it is an automorphism, and $\sigma\mapsto\alpha_\sigma$ is plainly multiplicative.
\end{proof}

\begin{remark}\label{rem:n3_reduction}
By this symmetry the whole nilpotency analysis of this subsection is governed by the case $n=3$, with the three copies $H=H^{\psi_0}$, $H^{\psi}=H^{\psi_1}$ and $H^{\psi_2}$. Indeed:
\begin{itemize}
\item[(i)] each pair of copies generates a subgroup $\langle H^{\psi_i},H^{\psi_j}\rangle$ which is an epimorphic image of $\chi(H)$ (Corollary~\ref{cor:embed}); since laws pass to quotients, the two-copy commutators $[H^{\psi_i},cH^{\psi_j}]$ are controlled exactly as for $\chi(H)$ in \cite{grs};
\item[(ii)] by Lemma~\ref{lem:relacoes_chi}(ii)--(iii), any further copy $H^{\psi_k}$ acts on $D_{ij}$ by conjugation as $H^{\psi_i}$ does, so that every commutator collapses to one within a single pair --- a phenomenon already displayed, in full, by the three copies $H^{\psi_0},H^{\psi_1},H^{\psi_2}$;
\item[(iii)] by Lemma~\ref{lem:sym_aut} all $\binom{n}{2}$ pairs are equivalent under $\operatorname{Aut}(\chi(n,H))$, and the terms of the lower central series of $\chi_n$ are products over these pairs; this is the source of the factor $\binom{n}{2}$ in Theorem~\hyperref[th:thmC]{C}(iii).
\end{itemize}
Accordingly, it suffices to understand the base copy $\chi(H)=\langle H,H^{\psi}\rangle$ together with the effect of a single extra copy; the results below are stated for general $n$ only to fix the notation.
\end{remark}

\begin{lemma}
\label{lem:relacoes_chi}  Let $x, \;y, \; z, \; y_i, \; z_i \; \in H$. Then
\begin{itemize}
\item[(i)] $[x^{\psi_i},y^{\psi_j}]=[x^{\psi_j},y^{\psi_i}], \; i, j \in \{
0, 1, \; \ldots, {n-1} \}$;

\item[(ii)] $[x^{\psi_i},y^{\psi_j}]^{z^{\psi_k}}=[x^{\psi_i},y^{\psi_j}]^{z^{\psi_i}}=[x^{\psi_i},y^{%
\psi_j}]^{z^{\psi_j}}, \; 0 \leq i \neq j \leq n-1, \; 0 \leq k \leq n-1$;

\item[(iii)] $[x^{\psi_i},y^{\psi_j}]^{w(z_1^{\epsilon_1},z_2^{\epsilon_2},%
\cdots,z_\ell^{\epsilon_\ell})}=[x^{\psi_i},y^{\psi_j}]^{w(z_1^{\delta_1},z_2^{%
\delta_2},\cdots,z_\ell^{\delta_\ell})}$ for $\epsilon_k \in \{\psi_t \mid 0 \leq t \leq n-1\}$, where $\delta_k =
\begin{cases}
\epsilon_k, & \text{if } \epsilon_k \in \{\psi_i, \psi_j\} \\
\psi_i, & \text{if } \epsilon_k \notin \{\psi_i, \psi_j\}
\end{cases}.$

\item[(iv)] $[x^{\psi_i},y^{\psi_j},x^{\psi_j}]=[x^{\psi_j},y^{\psi_j},x^{%
\psi_i}], \; 0 \leq i \neq j \leq n-1$;

\item[(v)] $[x^{\psi_i},y_1^{\psi_j},\cdots,y_k^{\psi_j},x^{\psi_j}]=[x^{%
\psi_j},y_1^{\psi_j},\cdots,y_k^{\psi_j},x^{\psi_i}], \; 0 \leq i \neq j \leq n-1$.
\end{itemize}
\end{lemma}

\begin{proof}
(i). By Definition~\ref{def:chin_1} we have that $[u^{\psi_i},u^{\psi_j}]=1$  for all  $u \in H, \;  i, j \in
\{0,1, \ldots,n-1\}$ and so,   
\begin{eqnarray*}
1&=& [x^{\psi_i}y^{-\psi_i},x^{\psi_j} y^{-\psi_j}] \\
&=& [x^{\psi_i},x^{\psi_j} y^{-\psi_j}]^{y^{-\psi_i}}[y^{-\psi_i},x^{\psi_j}
y^{-\psi_j}] \\
&=&[x^{\psi_i}, y^{-\psi_j}]^{y^{-\psi_i}}[y^{-\psi_i},x^{\psi_j}]^{
y^{-\psi_j}}.
\end{eqnarray*}
Conjugating the above equality by $y^{\psi_i}y^{\psi_j}$ we find that   
$$1=[x^{\psi_i}, y^{-\psi_j}]^{y^{\psi_j}}[y^{-\psi_i},x^{\psi_j}]^{
y^{\psi_i}},$$ 
from which we obtain $[x^{\psi_j},y^{\psi_i}]=[x^{\psi_i},y^{\psi_j}]$. \\
(ii). By the previous item we have  
\begin{eqnarray*}
[x^{\psi_i}z^{\psi_i},y^{\psi_j}]&=&[x^{\psi_j} z^{\psi_j},y^{\psi_i}] \\
{[x^{\psi_i},y^{\psi_j}]}^{z^{\psi_i}} [z^{\psi_i},y^{\psi_j}]&=&[x^{\psi_j}
,y^{\psi_i}]^{z^{\psi_j}}[ z^{\psi_j},y^{\psi_i}] \\
{[x^{\psi_i},y^{\psi_j}]}^{z^{\psi_i}} &=&[x^{\psi_j} ,y^{\psi_i}]^{z^{\psi_j}%
} \\
{[x^{\psi_i},y^{\psi_j}]}^{z^{\psi_i}} &=&{[x^{\psi_i},y^{\psi_j}]}%
^{z^{\psi_j}}.
\end{eqnarray*}
This yields $[x^{\psi_i},y^{\psi_j}]^{z^{\psi_i}}=[x^{\psi_i},y^{\psi_j}]^{z^{\psi_j}}$. Finally, for $k\notin\{i,j\}$, put $d=[x^{\psi_i},y^{\psi_j}]\in D_{ij}$ and $\ell=z^{-\psi_i}z^{\psi_k}=[z^{\psi_i},\psi_i^{-1}\psi_k]\in L_{ik}$. Since $\{i,k\}\neq\{i,j\}$, the relation $[D_{ij},L_{ik}]=1$ shows that $\ell$ commutes with both $d$ and $d^{z^{\psi_i}}\in D_{ij}$ (normality of $D_{ij}$), whence
$$[x^{\psi_i},y^{\psi_j}]^{z^{\psi_k}}=d^{z^{\psi_i}\ell}=\big(d^{z^{\psi_i}}\big)^{\ell}=d^{z^{\psi_i}}=[x^{\psi_i},y^{\psi_j}]^{z^{\psi_i}}.$$
(iii). Let $w(z_1^{\epsilon_1},z_2^{\epsilon_2},\cdots,z_\ell^{%
\epsilon_\ell})=h_1^{\psi_{i_1}}h_2^{\psi_{i_2}}\ldots h_m^{\psi_{i_m}}$ so
that 
\begin{equation*}
w(z_1^{\delta_1},z_2^{\delta_2},\cdots,z_\ell^{\delta_\ell})=h_1^{%
\delta_{i_1}}h_2^{\delta_{i_2}}\ldots h_m^{\delta_{i_m}},
\end{equation*}
where $\delta_k = \epsilon_k$ if  $\epsilon_k \in \{\psi_i, \psi_j\}$ and
$\delta_k = \psi_i$ otherwise. We proceed by induction on m. For $m=1$, item (ii) gives directly
\[[x^{\psi_i}, y^{\psi_j}]^{h_1^{\psi_{i_1}}} =
\begin{cases}
[x^{\psi_i}, y^{\psi_j}]^{h_1^{\psi_{i_1}}} \bigl(= [x^{\psi_i}, y^{\psi_j}]^{h_1^{\psi_i}}=[x^{\psi_i}, y^{\psi_j}]^{h_1^{\psi_j}}\bigr) & \text{if } \psi_{i_1} \in \{\psi_i, \psi_j\}, \\
[x^{\psi_i}, y^{\psi_j}]^{h_1^{\psi_i}} & \text{otherwise}.
\end{cases}
\]
Thus, $[x^{\psi_i}, y^{\psi_j}]^{h_1^{\psi_{i_1}}} = [x^{\psi_i}, y^{\psi_j}]^{h_1^{\delta_{i_1}}}.$
Suppose by induction that 
$$[x^{\psi_i},y^{\psi_j}]^{h_1^{\psi_{i_1}}h_2^{\psi_{i_2}}\ldots
h_m^{\psi_{i_m}}}=[x^{\psi_i},y^{\psi_j}]^{h_1^{\delta_{i_1}}h_2^{%
\delta_{i_2}}\ldots h_m^{\delta_{i_m}}}.$$ Using the same argument as above and the normality of $D_{ij}$, we finally obtain 
\begin{eqnarray*}
[x^{\psi_i},y^{\psi_j}]^{h_1^{\psi_{i_1}} h_2^{\psi_{i_2}} \ldots
h_m^{\psi_{i_m}} h_{m+1}^{\psi_{i_{m+1}}}}
&=&
\big(
[x^{\psi_i},y^{\psi_j}]^{h_1^{\delta_{i_1}} h_2^{\delta_{i_2}} \ldots
h_m^{\delta_{i_m}}}
\big)^{h_{m+1}^{\psi_{i_{m+1}}}} \\[6pt]
&=&
[x^{\psi_i},y^{\psi_j}]^{h_1^{\delta_{i_1}} h_2^{\delta_{i_2}} \ldots
h_m^{\delta_{i_m}} h_{m+1}^{\delta_{i_{m+1}}}}.
\end{eqnarray*}

(iv). Expanding the equality $[y^{\psi_j},x^{\psi_i}
x^{\psi_j}]=[y^{\psi_j},x^{\psi_j}x^{\psi_i}]$ we obtain  
\begin{equation*}
[y^{\psi_j},x^{\psi_j}][y^{\psi_j},x^{\psi_i}][y^{\psi_j},x^{\psi_i},x^{%
\psi_j}]=[y^{\psi_j},x^{\psi_i}][y^{\psi_j},x^{\psi_j}][y^{\psi_j},x^{%
\psi_j},x^{\psi_i}],
\end{equation*}
and thus,  
\begin{equation*}
[y^{\psi_j},x^{\psi_j}]^{-1}[y^{\psi_j},x^{\psi_i}]^{-1}[y^{\psi_j},x^{%
\psi_j}][y^{\psi_j},x^{\psi_i}][y^{\psi_j},x^{\psi_i},x^{\psi_j}]=[y^{%
\psi_j},x^{\psi_j},x^{\psi_i}],
\end{equation*}
while by the previous item we have  
\begin{equation*}
1=[y^{\psi_j},x^{\psi_i};y^{\psi_j},x^{\psi_i}]=
[y^{\psi_j},x^{\psi_i};y^{\psi_j},x^{\psi_j}], 
\end{equation*}
where $[y^{\psi_j},x^{\psi_i};y^{\psi_j},x^{\psi_i}]
:=[[y^{\psi_j},x^{\psi_i}],[y^{\psi_j},x^{\psi_i}]]$. Use of this relation gives   
\begin{equation*}
[y^{\psi_j},x^{\psi_i},x^{\psi_j}]=[y^{\psi_j},x^{\psi_j},x^{\psi_i}].
\end{equation*}
Therefore,   
\begin{eqnarray*}
{[y^{\psi_j},x^{\psi_i},x^{\psi_j}]}^{-[x^{\psi_i},y^{\psi_j}]} &=& {%
[y^{\psi_j},x^{\psi_j},x^{\psi_i}]}^{-[x^{\psi_i},y^{\psi_j}]} \\
{[x^{\psi_i},y^{\psi_j},x^{\psi_j}]} &\overset{\text{(iii)}}{=}& {%
[y^{\psi_j},x^{\psi_j},x^{\psi_i}]}^{-[x^{\psi_j},y^{\psi_j}]} \\
{[x^{\psi_i},y^{\psi_j},x^{\psi_j}]} &=& [x^{\psi_j},y^{\psi_j},x^{\psi_i}].
\end{eqnarray*}
(v). By inductive hypothesis on $k$ we can assume that  
\begin{equation*}
[x^{\psi_i},y_1^{\psi_j},\cdots,y_{k-1}^{\psi_j},x^{\psi_j}]=[x^{%
\psi_j},y_1^{\psi_j},\cdots, y_{k-1}^{\psi_j}, x^{\psi_i}], \; 0 \leq i \neq j \leq n-1.
\end{equation*}
Expanding the equality
\begin{equation*}
[x^{\psi_i},y_1^{\psi_j},\cdots,y_{k-2}^{\psi_j},y_{k-1}^{\psi_j}y_k^{%
\psi_j},x^{\psi_j}]=[x^{\psi_j},y_1^{\psi_j},\cdots,y_{k-2}^{%
\psi_j},y_{k-1}^{\psi_j}y_k^{\psi_j},x^{\psi_i}]
\end{equation*}
for $\; 0 \leq i \neq j \leq n-1$, we get  
\begin{eqnarray*}
&\bigl[\lbrack
x^{\psi_i},y_1^{\psi_j},\cdots,y_{k-2}^{\psi_j},y_k^{\psi_j}][x^{%
\psi_i},y_1^{\psi_j},\cdots,y_{k-2}^{\psi_j},y_{k-1}^{\psi_j}] \times \\ 
& \quad \times [x^{%
\psi_i},y_1^{\psi_j},\cdots,y_{k-2}^{\psi_j},y_{k-1}^{\psi_j},y_k^{%
\psi_j}],x^{\psi_j} \bigr]=& \\
&\bigl[\lbrack x^{\psi_j},y_1^{\psi_j},\cdots,y_{k-2}^{\psi_j},y_k^{\psi_j}][x^{%
\psi_j},y_1^{\psi_j},\cdots,y_{k-2}^{\psi_j}
,y_{k-1}^{\psi_j}] \times & \\ 
& \quad \times[x^{\psi_j},y_1^{\psi_j},\cdots,y_{k-2}^{\psi_j},y_{k-1}^{%
\psi_j},y_k^{\psi_j}],x^{\psi_i} \bigr].&
\end{eqnarray*}
Therefore,  
\begin{eqnarray*}
&[x^{\psi_i},y_1^{\psi_j},\cdots,y_{k-2}^{\psi_j},y_k^{\psi_j},x^{%
\psi_j}]^{[x^{\psi_i}
,y_1^{\psi_j},\cdots,y_{k-2}^{\psi_j},y_{k-1}^{\psi_j}][x^{\psi_i},y_1^{%
\psi_j} ,\cdots,y_{k-2}^{\psi_j},y_{k-1}^{\psi_j},y_k^{\psi_j}]}& \\
& \times
[x^{\psi_i},y_1^{\psi_j},\cdots,y_{k-2}^{\psi_j},y_{k-1}^{\psi_j},x^{%
\psi_j}]^{[x^{\psi_i},y_1^{\psi_j},\cdots,y_{k-2}^{\psi_j},y_{k-1}^{%
\psi_j},y_k^{\psi_j}]} & \\
& \times
[x^{\psi_i},y_1^{\psi_j},\cdots,y_{k-2}^{\psi_j},y_{k-1}^{\psi_j},y_k^{%
\psi_j},x^{\psi_j}]& \\
&=[x^{\psi_j},y_1^{\psi_j},\cdots,y_{k-2}^{\psi_j},y_k^{\psi_j},x^{%
\psi_i}]^{[x^{\psi_j},y_1^{\psi_j},\cdots,y_{k-2}^{\psi_j},y_{k-1}^{%
\psi_j}][x^{\psi_j},y_1^{\psi_j},\cdots,y_{k-2}^{\psi_j},y_{k-1}^{%
\psi_j},y_k^{\psi_j}]}& \\
& \times
[x^{\psi_j},y_1^{\psi_j},\cdots,y_{k-2}^{\psi_j},y_{k-1}^{\psi_j},x^{%
\psi_i}]^{[x^{\psi_j},y_1^{\psi_j},\cdots,y_{k-2}^{\psi_j},y_{k-1}^{%
\psi_j},y_k^{\psi_j}]} & \\
& \times
[x^{\psi_j},y_1^{\psi_j},\cdots,y_{k-2}^{\psi_j},y_{k-1}^{\psi_j},y_k^{%
\psi_j},x^{\psi_i}].&
\end{eqnarray*}
By the inductive hypothesis and item (iii) we finally conclude that  
\begin{equation*}
[x^{\psi_i},y_1^{\psi_j},\cdots,y_{k}^{\psi_j},x^{\psi_j}]=[x^{\psi_j},y_1^{%
\psi_j},\cdots,y_{k}^{\psi_j},x^{\psi_i}], \; 0 \leq i \neq j \leq n-1.
\end{equation*}
The proof of Lemma~\ref{lem:relacoes_chi} is complete.
\end{proof}

\begin{lemma}\label{lem:relnil_chi}
\begin{itemize}
\item[(i)] $[H^{\psi_i},H^{\psi_j},H^{\epsilon_1},\ldots,H^{%
\epsilon_{\ell}}]=[H^{\psi_i},(\ell+1)H^{\psi_j}]$  for all $\ell \geq 1, \; 0 \leq i \neq j \leq n-1$, and any $\; \epsilon_k \in \{\psi_t \mid 0\leq t \leq n-1\}$;

\item[(ii)] $[H^{\psi_i},mH^{\psi_j};\gamma_{\ell}(H^{\psi_j})] \leq
[H^{\psi_i},(m+\ell)H^{\psi_j}]$,  for all  $m \geq 0, \; \ell \geq 1, \; 0 \leq i \neq j \leq n-1$;

\item[(iii)] $\displaystyle [H^{\epsilon_1},H^{\epsilon_2},\ldots, H^{\epsilon_{\ell}}]\leq
[H^{\psi_i},(\ell-1)H^{\psi_j}]\,\gamma(H^{\psi_i})^{\chi_n}\gamma(H^{\psi_j})^{\chi_n}$, \\
for $\epsilon_k \in \{\psi_i,\psi_j\}, \; i\neq j$.
\end{itemize}
\end{lemma}

\begin{proof}
Item (i) follows from Lemma~\ref{lem:relacoes_chi}: by parts (ii)--(iii), each factor $H^{\epsilon_k}$ acts on $D_{ij}=[H^{\psi_i},H^{\psi_j}]$ by conjugation as $H^{\psi_i}$ does, which reduces the left-hand side to the case $\epsilon_k\in\{\psi_i,\psi_j\}$; this equals $[H^{\psi_i},(\ell+1)H^{\psi_j}]$ by parts (iv)--(v).

(ii). We argue by induction on $\ell$. For $\ell=1$ there is nothing to do. By inductive
hypothesis suppose the result valid for $1 \leq k \leq \ell-1$. Taking $%
x \in \gamma_{\ell-1}(H^{\psi_j}), \; y \in H^{\psi_j}, \; z \in
[H^{\psi_i},mH^{\psi_j}]$, and considering an equivalent form of Witt's
identity $[z, [x,y]]=[z,y^{-1},x^z]^y[z,x^{-1},y^{-1}]^{xy}$ and Lemma~\ref{lem:relacoes_chi}, we get  
\begin{eqnarray*}
[H^{\psi_i},mH^{\psi_j},\gamma_\ell(H^{\psi_j})] & \leq &
[H^{\psi_i},(m+1)H^{\psi_j},
\gamma_{\ell-1}(H^{\psi_j})]^{H^{\psi_j}} \times  \\ 
\ & \quad &  \times [H^{\psi_i},mH^{\psi_j},\gamma_{\ell}(H^{%
\psi_j})]^{H^{\psi_j}} \\
&\leq&[H^{\psi_i},(m+1)H^{\psi_j},
\gamma_{\ell-1}(H^{\psi_j})][H^{\psi_i},mH^{\psi_j},\gamma_{\ell}(H^{\psi_j})] \\
&\leq&[H^{\psi_i},(m+\ell)H^{\psi_j}].
\end{eqnarray*}

(iii). If $\ell \geq 2$ and $(\epsilon_1,\ldots,
\epsilon_\ell)=(\psi_i,\psi_i,\ldots,\psi_i)$ or $(\epsilon_1,\ldots, 
\epsilon_\ell)=(\psi_j,\psi_j,\ldots,\psi_j)$, then there is nothing to do. Thus,
WLOG we can assume that  
\begin{equation*}
(\epsilon_1,\ldots,\epsilon_\ell)=(\psi_i,\psi_i,\ldots,\psi_i,\psi_j,%
\epsilon_{k+2}, \ldots, \epsilon_\ell)
\end{equation*}
for some $1 \leq k \leq \ell$. Therefore,  
\begin{eqnarray*}
[H^{\epsilon_1},H^{\epsilon_2},\ldots,
H^{\epsilon_\ell}]&=&[\gamma_k(H^{\psi_i}),H^{\psi_j},H^{\epsilon_{k+2}},%
\ldots, H^{\epsilon_\ell}] \\
&=&[\gamma_k(H^{\psi_i}),H^{\psi_j},(\ell-k-1)H^{\psi_j}] 
\quad (\text{Lemma~\ref{lem:relacoes_chi}, part (iii)}) \\
&=&[\gamma_k(H)^{\psi_i},H^{\psi_j},(\ell-k-1)H^{\psi_j}] \\
&=&[\gamma_k(H)^{\psi_j},H^{\psi_i},(\ell-k-1)H^{\psi_j}] 
\quad \text{(Lemma~\ref{lem:relacoes_chi}, part(i))} \\
&=&[H^{\psi_i},\gamma_k(H^{\psi_j}),(\ell-k-1)H^{\psi_j}] \\
&=&[H^{\psi_i},\gamma_{\ell-1}(H^{\psi_j})] \qquad \text{(By item (ii))}.
\end{eqnarray*}
This finishes the proof of Lemma~\ref{lem:relnil_chi}.
\end{proof}

\begin{lemma}
\label{lem:comut_chi}  Let $\gamma_{c+1}(H)=1$. Then
\begin{itemize}
\item[(i)] $[\gamma_{c+1}(\chi_n),\gamma_2(\chi_n)]=1$;
\item[(ii)] $[\gamma_i(\chi_n),\gamma_2(\chi_n),(2c-1-i)\chi_n]=1, \;
\; \text{for all} \; i \geq 2.$
\end{itemize}
\end{lemma}

\begin{proof}
(i). By Lemma~\ref{lem:relnil_chi} we have  $%
\gamma_{c+1}(\chi_n)\leq \prod_{0 \leq i<j\leq n-1}{[H^{\psi_i},cH^{%
\psi_j}]}$. Thus,   
\begin{eqnarray*}
[\gamma_{c+1}(\chi_n),\gamma_2(\chi_n)]&\leq& \Bigl[\prod_{0 \leq
i<j\leq n-1}{[H^{\psi_i},cH^{\psi_j}]} ,\prod_{0 \leq i\leq j\leq n-1}{%
[H^{\psi_i},H^{\psi_j}]}\Bigr] \\
&\leq &\prod_{0 \leq i< j\leq n-1}\Bigl[{[H^{\psi_i},cH^{\psi_j}]} ,{%
[H^{\psi_i},H^{\psi_j}]}\Bigr] \\
&& \text{(by Lemma~\ref{lem:relacoes_chi}(ii)--(iii) and Lemma~\ref{lem:relnil_chi}(iii))} \\
&=& \prod_{0 \leq i< j\leq n-1}\Bigl[{[H^{\psi_i},H^{\psi_j}]},{%
[H^{\psi_i},cH^{\psi_j}]} \Bigr] \\
&=& \prod_{0 \leq i< j\leq n-1}\Bigl[{[H^{\psi_i},H^{\psi_j}]},{%
[H^{\psi_j},cH^{\psi_j}]} \Bigr] \\
&=& 1.
\end{eqnarray*}

(ii). Using the Three Subgroups Lemma we
obtain for $i \geq 2:$ 
\begin{equation*}
[\gamma_i(\chi_n),\gamma_2(\chi_n),(2c-1-i)\chi_n] \leq \prod_{
\ell+m=2c+1, \; \ell, m \geq 2} {[\gamma_\ell(\chi_n),\gamma_m(\chi_n)]}.
\end{equation*}
For $m \geq c+1$ or $\ell \geq c+1$, the result follows by the
first part.
\end{proof}

\begin{lemma}
\label{lem:permut_chi} Let $\gamma_{c+1}(H)=1$. Then,
\begin{equation*}
[h_1^{\epsilon_1},h_2^{\epsilon_2},
\ldots,h_{2c+1}^{\epsilon_{2c+1}}]=[h_1^{\epsilon_1},h_2^{\epsilon_2},h_{%
\sigma(3)}^{\epsilon_{\sigma(3)}},
\ldots,h_{\sigma(2c+1)}^{\epsilon_{\sigma(2c+1)}}]
\end{equation*}
for all $h_i \in H$, $\epsilon_i \in \{\psi_i,\psi_j\}$ and any permutation $\sigma$ of  $\{3,4, \ldots , 2c+1\}.$
\end{lemma}

\begin{proof}
We have, for all $i \in \{3,\ldots, 2c+1\}$, that
\begin{equation*}
[h_1^{\epsilon_1},h_2^{\epsilon_2},
\ldots,[h_{i-1}^{\epsilon_{i-1}},h_{i}^{\epsilon_{i}}]]\equiv 1 \pmod{[\gamma_{i-2}(\chi_n),\gamma_2(\chi_n)]}, 
\end{equation*}
and thus, by Levin~\cite[Theorem 2.1]{Levin} we obtain
\begin{equation*}
[h_1^{\epsilon_1},h_2^{\epsilon_2},
\ldots,h_i^{\epsilon_i},h_{i-1}^{\epsilon_{i-1}}]\equiv
[h_1^{\epsilon_1},h_2^{\epsilon_2},
\ldots,h_{i-1}^{\epsilon_{i-1}},h_{i}^{\epsilon_{i}}] \pmod{[\gamma_{i-2}(\chi_n),\gamma_2(\chi_n)]}.
\end{equation*}
Now, by Lemma~\ref{lem:comut_chi} item (ii),  
\begin{equation*}
[h_1^{\epsilon_1},h_2^{\epsilon_2},
\ldots,h_i^{\epsilon_i},h_{i+1}^{\epsilon_{i+1}},\ldots,h_{(2c+1)}^{%
\epsilon_{(2c+1)}}]= [h_1^{\epsilon_1},h_2^{\epsilon_2},
\ldots,h_{i+1}^{\epsilon_{i+1}},h_{i}^{\epsilon_{i}},\ldots,h_{(2c+1)}^{%
\epsilon_{(2c+1)}}]. 
\end{equation*}
Therefore,  
\begin{equation*}
[h_1^{\epsilon_1},h_2^{\epsilon_2},
\ldots,h_{2c+1}^{\epsilon_{2c+1}}]=[h_1^{\epsilon_1},h_2^{\epsilon_2},h_{%
\sigma(3)}^{\epsilon_{\sigma(3)}},
\ldots,h_{\sigma(2c+1)}^{\epsilon_{\sigma(2c+1)}}] 
\end{equation*}
for all permutations $\sigma$ of $\{3,4, \ldots , 2c+1\}$.
\end{proof}

\begin{lemma}
\label{lem:nil_chi}  If $H$ is a locally nilpotent group, then $\chi(n,H)$ is also locally nilpotent.
\end{lemma}

\begin{proof}
Let $g_1,g_2, \ldots,g_r$ be elements of $\chi(n,H)$ and $%
\{h_1,h_2,\ldots,h_m\}$ be their support in $H$. Consider $m \geq 2$ and assume that the subgroup 
$\langle h_1,h_2, \ldots, h_m \rangle \leq H$ is nilpotent, of class at most $c$, say. To show that the subgroup $\langle g_1,g_2,\ldots,g_r \rangle \leq \chi(n,H)$ is nilpotent, by Lemma~\ref{lem:relnil_chi} it 
suffices to show that 
\begin{equation*}
[x^{\psi_i},y^{\psi_j},z_1^{\psi_j},\ldots ,z_{c^*}^{\psi_j}]=1 
\end{equation*}
for some $c^*> c$ and all $\; 0 \leq i \neq j \leq n-1, \;x,y,z_i
\in \langle h_1, \ldots, h_m \rangle$.  To this end, let $c^* \geq 2 c m$; by Lemmas \ref{lem:comut_chi} and  \ref{lem:permut_chi},
 the commutator $[x^{\psi_i},y^{\psi_j},z_1^{\psi_j},\ldots
,z_{c^*}^{\psi_j}] $ can be rewritten as a product of commutators of the form  
\begin{equation*}
[x^{\psi_i},y^{\psi_j},k_1{h^{\prime }}_1^{\psi_j},\ldots ,k_{m}{h^{\prime }}%
_m^{\psi_j}], 
\end{equation*}
where $\{{h^{\prime }_1}, \ldots, {h^{\prime }}_m \}=\{{h_1}, \ldots, {h}_m
\}$, $k_1\geq \ldots \geq k_m \geq 0$ and $\sum_{i=1}^{m}k_i \geq c^* \geq 2cm$%
. Now we have that $k_1 \geq 2c$, so it suffices to show that $[x^{\psi_i},y^{%
\psi_j},kz^{\psi_j}]=1$ for all $k \geq 2c$ and $x,y , z \; \in \;
\langle h_1,\ldots, h_m \rangle$.  Let $\overline{H}=\langle x,y,z\rangle$.
Using the Jacobi congruence and the hypothesis
that $\gamma_{c+1}(\overline{H}) = 1$, we can write   
\begin{equation*}
[x^{\psi_i},y^{\psi_j},z^{\psi_j}]\equiv[x^{\psi_i},z^{\psi_j},y^{\psi_j}][{%
x^{\psi_i},[y^{\psi_j},z^{\psi_j}]}] \; \text{mod} \;
(\gamma_2(\gamma_2(\langle x^{\psi_i},y^{\psi_j},z^{\psi_j}\rangle ))).
\end{equation*}
By  Lemma~\ref{lem:comut_chi} we have   
\begin{multline*}
[x^{\psi_i},y^{\psi_j},z^{\psi_j},(k-1)z^{\psi_j}] \\
=[x^{\psi_i},z^{\psi_j},y^{\psi_j},(k-1)z^{\psi_j}]\,[x^{\psi_i},[y^{\psi_j},z^{\psi_j}],(k-1)z^{\psi_j}].
\end{multline*}
and
\begin{eqnarray}  \label{eq:jacobi-mod}
[x^{\psi_i},[y^{\psi_j},z^{\psi_j}]]\equiv[z^{\psi_j},y^{\psi_j},x^{\psi_i}]%
\; \text{mod}\; (\gamma_2(\gamma_2(\langle x^{\psi_i},y^{\psi_j},z^{\psi_j} \rangle))).
\end{eqnarray}
Using \eqref{eq:jacobi-mod} and Lemma~\ref{lem:comut_chi} again, we have  
\begin{eqnarray*}
{\ [x^{\psi_i},[y^{\psi_j},z^{\psi_j}],(k-1)z^{\psi_j}]}&=&[z^{\psi_j},y^{%
\psi_j},x^{\psi_i},(k-1)z^{\psi_j}] \\
&=&[z^{\psi_j},y^{\psi_j},(k-1)z^{\psi_j},x^{\psi_i}] \quad (\text{ Lemma %
\ref{lem:permut_chi}}) \\
&=&1.
\end{eqnarray*}
Therefore,  
\begin{align*}
[x^{\psi_i},y^{\psi_j},z^{\psi_j},(k-1)z^{\psi_j}] &= [x^{\psi_i},z^{\psi_j},y^{\psi_j},(k-1)z^{\psi_j}] \\
&= [z^{\psi_i},x^{\psi_j},y^{\psi_j},(k-2)z^{\psi_j},z^{\psi_j}]^{-1} \\
&= [z^{\psi_j},x^{\psi_j},y^{\psi_j},(k-2)z^{\psi_j},z^{\psi_i}]^{-1} \quad \text{(Lemma~\ref{lem:relacoes_chi}, (iv))} \\
&= 1,
\end{align*}
which proves Lemma~\ref{lem:nil_chi}.
\end{proof}

We close this subsection by establishing Theorem~\hyperref[th:thmC]{C} for $\chi(n,H)$, first obtained in \cite{oliveira-thesis}.

\begin{theorem}\label{th:chinasthngnil_chi}
Let $H$ be an $m$-generated nilpotent group of class at most $c$, with $m \geq 2$ and $c\geq 1$.
\begin{itemize}
    \item[(i)] If $m = 2$, then $\gamma_{c+2}(\chi(n,H))=1$;
    \item[(ii)] If $2 < m \leqslant c+2$, then $\gamma_{c+3}(\chi(n,H))=1$;
    \item[(iii)] If  $m \geqslant c+3$, then  $\gamma_{c+3}(\chi(n,H))$ is an elementary abelian $2$-group of rank at most
\begin{equation*}
\sum_{k=c+3}^{m}\binom{m}{k}\binom{n}{2}.
\end{equation*}
\end{itemize}
\end{theorem}

\begin{proof}
Item (i). Since $H$ is a nilpotent group of class at most $c \geq 1$, by Lemma~\ref%
{lem:relacoes_chi} (v) the group $\chi(n,H)$ satisfies the identity
\begin{equation}  \label{eq:ident}
[x^{\psi_i},y_2^{\psi_j},\ldots y_{c+1}^{\psi_j},x^{\psi_j}]=1
\end{equation}
for all $x, \;y_i \; \in H$. Moreover, if $H=\langle x,y \rangle$ then, modulo $%
\gamma_{c+3}(\chi(n,H))$ we see that $\gamma_{c+2}(\chi(n,H))$ is generated by
elements of the form $[x^{\psi_i},z_2^{\psi_j},\ldots
z_{c+1}^{\psi_j},x^{\psi_j}]$, where $z_i \in \{x,y \}$ and $0 \leq i \neq j \leq n-1$, which are trivial by~\eqref{eq:ident}.
Thus,  $\gamma_{c+3}(\chi(n,H))= \gamma_{c+2}(\chi(n,H))$ and so, being $%
\chi(n,H)$ a nilpotent group, it follows that $\gamma_{c+2}(\chi(n,H))=1$. 

\noindent
Item (ii). In this case, where $\gamma _{c+1}(H)=1$, identity~\eqref{eq:ident}
modulo $\gamma _{c+3}(\chi(n,H))$ gives us
\begin{equation}
1\equiv \lbrack x^{\psi_i}y_{1}^{\psi_j},\ldots ,y_{c+1}^{\psi_{j}},x^{\psi_j}y_{1}^{\psi_j}]\equiv \lbrack x^{\psi
^{i}},y_{2}^{\psi_j},\ldots ,y_{c+1}^{\psi_j},y_{1}^{\psi_j}][y_{1}^{\psi_i},y_{2}^{\psi_j}\ldots ,y_{c+1}^{\psi_j},x^{\psi_j}].  \label{eq:indent2}
\end{equation}%
Taking the commutator with $x^{\psi_j}$ and using~\eqref{eq:ident}
 we obtain
\begin{equation}
\lbrack y_{1}^{\psi_i},y_{2}^{\psi_j}\ldots ,y_{c+1}^{\psi_{j}},x^{\psi_j},x^{\psi_j}]\equiv 1\quad (\text{mod}\,\gamma
_{c+4}(\chi(n,H)))  \label{eq:identmodgama4}
\end{equation}%
for all $x,\;y_{i}\;\in H$. Moreover, modulo $\gamma _{c+4}(\chi(n,H))$, for $2\leq k\leq c$ we have 
\begin{tabular}{rcl}
\multicolumn{3}{l}{$\bigl[\lbrack y_{1}^{\psi_i},y_{2}^{\psi_j},\ldots
,y_{k}^{\psi_j}],[x^{\psi_j},y_{k+1}^{\psi_j}],y_{k+2}^{\psi_{j}},\ldots ,y_{c+1}^{\psi_j},x^{\psi_j}\bigr]$} \\ 
& $\equiv $ & $\bigl[\lbrack y_{1}^{\psi_i},y_{2}^{\psi_j},\ldots
,y_{k}^{\psi_j}],[x^{\psi_i},y_{k+1}^{\psi_j}],y_{k+2}^{\psi
^{j}},\ldots ,y_{c+1}^{\psi_j},x^{\psi_j}\bigr]$ \\ 
& $\equiv $ & $\bigl[\lbrack x^{\psi_i},y_{k+1}^{\psi_j}],[y_{1}^{\psi_{j}},y_{2}^{\psi_j},\ldots ,y_{k}^{\psi_j}],y_{k+2}^{\psi_{j}},\ldots ,y_{c+1}^{\psi_j},x^{\psi_j}\bigr]^{-1}$ \\ 
& $\equiv $ & $\bigl[\lbrack x^{\psi_j},y_{k+1}^{\psi_j}],[y_{1}^{\psi
^{j}},y_{2}^{\psi_j},\ldots ,y_{k}^{\psi_j}],y_{k+2}^{\psi
^{j}},\ldots ,y_{c+1}^{\psi_j},x^{\psi_i}\bigr]^{-1}$ \quad \text{(by eq.~\ref{eq:ident})} \\ 
& $\equiv$ & 1.
\end{tabular}
\newline
We record this as 
\begin{equation}
\bigl[\lbrack y_{1}^{\psi_i},y_{2}^{\psi_j},\ldots ,y_{k}^{\psi_{j}}],[x^{\psi_j},y_{k+1}^{\psi_j}],y_{k+2}^{\psi_j},\ldots
,y_{c+1}^{\psi_j},x^{\psi_j}\bigr]\equiv 1\;\;\text{mod}\;\gamma
_{c+4}(\chi(n,H))  \label{eq:indentcomutmod4}
\end{equation}%
for all $x,\;y_{i}\;\in H,\;0\leq i<j\leq \;n-1$ and $2\leq k\leq c$. Then, by
equation \ref{eq:indentcomutmod4}, for $2\leq k\leq c$ we have\newline
\begin{tabular}{rcl}
\multicolumn{3}{l}{$\lbrack y_{1}^{\psi_i},y_{2}^{\psi_j},\ldots
,y_{k}^{\psi_j},x^{\psi_j},y_{k+1}^{\psi_j},y_{k+2}^{\psi_{j}},\ldots ,y_{c+1}^{\psi_j},x^{\psi_j}]$} \\ 
& $\equiv $ & $[y_{1}^{\psi_i},y_{2}^{\psi_j},\ldots ,y_{k}^{\psi_{j}},x^{\psi_j},y_{k+1}^{\psi_j},y_{k+2}^{\psi_j},\ldots
,y_{c+1}^{\psi_j},x^{\psi_j}]$ \\ 
& $\equiv $ & $[y_{1}^{\psi_i},y_{2}^{\psi_j},\ldots ,y_{k+1}^{\psi_{j}},x^{\psi_j},y_{k+2}^{\psi_j},\ldots ,y_{c+1}^{\psi_j},x^{\psi_{j}}]$ \\ 
& $\vdots $ &  \\ 
& $\equiv $ & $[y_{1}^{\psi_i},y_{2}^{\psi_j},\ldots ,y_{c+1}^{\psi
^{j}},x^{\psi_j},x^{\psi_j}]$ \\ 
& $\equiv $ & 1 \, (by eq.~\ref{eq:identmodgama4}).
\end{tabular}

\noindent
Moreover, by~\eqref{eq:ident} we get
\begin{equation*}
\lbrack y_{1}^{\psi_i},x^{\psi_j},y_{2}^{\psi_j},\ldots
,y_{c+1}^{\psi_j},x^{\psi_j}]\equiv \lbrack x^{\psi_i},y_{1}^{\psi
^{j}},y_{2}^{\psi_j},\ldots ,y_{c+1}^{\psi_j},x^{\psi
^{j}}]^{-1}\equiv 1.
\end{equation*}%
Thus,  
\begin{equation}
\lbrack y_{1}^{\psi_i},y_{2}^{\psi_j},\ldots ,y_{k}^{\psi_{j}},x^{\psi_j},y_{k+1}^{\psi_j},y_{k+2}^{\psi_j},\ldots
,y_{c+1}^{\psi_j},x^{\psi_j}]\equiv 1 \pmod{\gamma _{c+4}(\chi(n,H))},  \label{eq:identgeral}
\end{equation}%
for all $1\leq k\leq c+1$. Replacing $x$ by $xz$ in equation \ref%
{eq:identgeral} and expanding modulo $\gamma _{c+4}(\chi(n,H))$,
the following congruence holds:
\begin{eqnarray}
\lbrack y_{1}^{\psi_i},y_{2}^{\psi_j},\ldots ,y_{k}^{\psi
^{j}},x^{\psi_j},y_{k+1}^{\psi_j},y_{k+2}^{\psi_j},\ldots
,y_{c+1}^{\psi_j},z^{\psi_j}] &\equiv &  \label{eq:identgeral2} \\
\lbrack y_{1}^{\psi_i},y_{2}^{\psi_j},\ldots ,y_{k}^{\psi
^{j}},z^{\psi_j},y_{k+1}^{\psi_j},y_{k+2}^{\psi_j},\ldots
,y_{c+1}^{\psi_j},x^{\psi_j}]^{-2}. &&  \notag
\end{eqnarray}%
Using  this last congruence we
obtain that all commutators of degree $c+3$ in $\chi(n,H)$ with a
double entry $x$ can be expressed, modulo $\gamma _{c+4}(\chi(n,H))$, 
 as a product of commutators of the form 
\begin{equation*}
\lbrack y_{1}^{\psi_i},y_{2}^{\psi_j},\ldots ,y_{k}^{\psi
^{j}},x^{\psi_j},y_{k+1}^{\psi_j},y_{k+2}^{\psi_j},\ldots
,y_{c+1}^{\psi_j},x^{\psi_j}],\quad 1\leq k\leq c+1,
\end{equation*}%
which are trivial by~\eqref{eq:identgeral}. In particular, being $H$ $m$-generated  
with $\gamma _{c+1}(H)=1$ and $m\leq c+2$, we find that $\gamma _{m+3}(\chi(n,H))=\gamma _{m+4}(\chi(n,H))=\ldots =1$. Therefore, by 
 Lemma~\ref{lem:permut_chi}, item (ii) is finally proven.

\noindent
As for item (iii), considering again that $H$ is nilpotent of class at most $c$, congruence~\eqref{eq:identgeral2} leads to
\begin{eqnarray*}
[y_1^{\psi_i},\ldots,y_{c+1}^{\psi_j},x^{\psi_j},z^{\psi_j}] & \equiv
&[y_1^{\psi_i},\ldots,y_{c+1}^{\psi_j},z^{\psi_j},x^{\psi_j}]^{-1} \\
& \overset{\text{(Lemma~\ref{lem:comut_chi})}}{ \equiv} &
[y_1^{\psi_i},\ldots,y_{c+1}^{\psi_j},x^{\psi_j},z^{\psi_j}]^{-1},
\end{eqnarray*}
from what we get   
\begin{equation*}
[y_1^{\psi_i},\ldots,y_{c+1}^{\psi_j},x^{\psi_j},z^{\psi_j}]^2 \equiv 1
\quad \text{mod} \; \gamma_{c+4}(\chi(n,H)). 
\end{equation*}
By item (ii), all commutators of degree $c+4$ in $\chi(n,H)$
with entries in the set $\{y_1,\ldots,y_{c+1},x,z\}$ are trivial. Therefore,   
\begin{equation}  \label{eq:identquad}
[y_1^{\psi_i},\ldots,y_{c+1}^{\psi_j},x^{\psi_j},z^{\psi_j}]^2 \equiv 1.
\end{equation}
Repeating the application of congruence~\eqref{eq:identgeral2} it follows that
\begin{equation*}
[y_1^{\psi_i},y_2^{\psi_j}\ldots,y_{c+3}^{\psi_j}]=[y_1^{\psi_i},y_{\rho(
2)}^{\psi_j}\ldots,y_{\rho(c+3)}^{\psi_j}]^{|\rho|},
\end{equation*}
where $\rho$ is a permutation of $\{2,3,\ldots, c+3 \}$ and $|\rho|$ stands for the
sign of $\rho$. So, if $H$ is an $m$-generated group with $m \geq c+3$ then, 
for $c+3\leq k \leq m$, there exist $\binom{m}{k}\binom{n}{2}$ distinct commutators of degree $k$. Therefore, by congruence~\eqref{eq:identquad} and the relations $[D_{ij}, D_{ks}]=1$ for all $0\leq i<j\leq n-1$ and $0\leq k < s \leq n-1$, we have proven item (iii). The proof of Theorem~\ref{th:chinasthngnil_chi} is now complete. 
\end{proof}

\section{The Group $\protect\chi_3(H)$} \label{sec:chi_n_p}

In this section, we show that imposing relations of the form $[D,L] = 1$ as we did in the definition of the group $\chi(n, H)$ is necessary if we expect constructions involving multiple copies of $H$ to inherit the finiteness property of $H$, as in the classical construction. Focusing only on the relations expressing weak commutativity, considering as before $n \geq 3$ isomorphic copies of $H$ through the isomorphisms  
$\psi_i : H \longrightarrow H^{\psi_i}, \ 0 \leq i \leq n-1,$
with $\psi_0 = 1_H$ and $\psi_1 = \psi$,  recall the mother group 
\[
\chi_n(H) := \left\langle H, H^{\psi}, \ldots, H^{\psi_{n-1}} \;\middle|\; [h^{\psi_i}, h^{\psi_j}] = 1, \ \forall h \in H, \ i \neq j \right\rangle.
\]
For $H$ an elementary abelian 3-group of rank 3, we suspect that $\chi_3(H)$ is infinite, whereas for rank 2 we obtain the following results.

\begin{lemma}\label{lem:kernel_general}
Let $p$ be a prime, $H = \langle a, b \rangle \cong C_{p^t} \times 
C_{p^u}$ a $2$-generated abelian $p$-group with $t \geq u \geq 1$, 
so that $\exp(H) = p^t$. Let 
$\pi \colon \chi_3(H) \twoheadrightarrow \chi(3,H)$ be the natural 
epimorphism and $K = \ker\pi$. Then:
\begin{itemize}
    \item[(i)] For $0 \leq i < j \leq 2$, the subgroups
    $D_{ij}$ and $L_{ij}$ of $\chi_3(H)$ are given by
    \[
        D_{ij} = \langle [a^{\psi_i}, b^{\psi_j}] \rangle,
        \quad
        L_{ij} = \langle a^{-\psi_i}a^{\psi_j},\;
                          b^{-\psi_i}b^{\psi_j},\;
                          [a^{\psi_i}, b^{\psi_j}] \rangle.
    \]
    The generator $[a^{\psi_i}, b^{\psi_j}]$ of $D_{ij}$ has
    order dividing $p^{\min\{t,u\}} = p^u$. The generators
    $a^{-\psi_i}a^{\psi_j}$ and $b^{-\psi_i}b^{\psi_j}$ of
    $L_{ij}$ have orders dividing $\mathrm{ord}(a) = p^t$ and
    $\mathrm{ord}(b) = p^u$ respectively. In particular, all
    generators have order dividing $\exp(H) = p^t$.

    \item[(ii)] $K$ is generated as a normal subgroup of
    $\chi_3(H)$ by the four double commutators
    \[
        k_1 = [a,\, \bp,\, \app], \quad
        k_2 = [a,\, \bp,\, \bpp], \quad
        k_3 = [a,\, \bpp,\, \ap], \quad
        k_4 = [\ap,\, \bpp,\, a];
    \]
    that is, $K = \langle k_1, k_2, k_3, k_4
    \rangle^{\chi_3(H)}$.

    \item[(iii)] If $p$ is odd, then $K \leq Z(\chi_3(H))$
    (hence $K$ is abelian), $k_1 = k_3 = k_4$, and
    \[
        k_2 = [a,\, \bpp,\, \bp]
            = [\ap,\, \bpp,\, b],
    \]
    so that $K = \langle k_1,\, k_2 \rangle$ has rank at most
    $2$ and $\exp(K)$ divides $p^u = p^{\min\{t,u\}}$.
    \item[(iv)] If $p = 2$, then $K$ is non-central in
    general (so the $p$ odd reduction (Step~4) does not
    apply); the four generators $k_1, k_2, k_3, k_4$ are
    linearly independent in $K^{\mathrm{ab}}$, so in
    particular $K^{\mathrm{ab}}$ has $\mathbb{Z}$-rank
    at least $4$; and the generators $k_i$ may have
    infinite order in $\chi_3(H)$ (see Theorem~\ref{thm:chi_n_klein} for the case $H=C_2\times C_2$).
\end{itemize}
\end{lemma}

\begin{proof}
\textbf{(i).} By \cite[Proposition~4.1.6(i)]{sidki-wp}, since $H$ is a $2$-generated abelian group, each $D_{ij}$ is central in $\langle H^{\psi_i}, H^{\psi_j}\rangle$; hence commutators between these two copies are bilinear. For each pair $i < j$, the subgroup 
$\langle H^{\psi_i}, H^{\psi_j}\rangle \leq \chi_3(H)$ 
is a homomorphic image of $\chi(H)$ via the map 
$h \mapsto h^{\psi_i}$, $h^{\psi} \mapsto h^{\psi_j}$. 
Since $H$ is abelian, the structure of $D_{ij}$ and 
$L_{ij}$ inside $\langle H^{\psi_i}, H^{\psi_j}\rangle$ 
follows from \cite[Section~4]{sidki-wp}. More precisely, 
the quotients
\[
    \frac{\langle H^{\psi_i}, H^{\psi_j}\rangle}{D_{ij}} 
    \cong H \times H,
    \qquad
    \frac{\langle H^{\psi_i}, H^{\psi_j}\rangle}{L_{ij}} 
    \cong H,
\]
together with the fact that $H = \langle a, b\rangle$ is 
$2$-generated and abelian, force
\[
    D_{ij} = \bigl\langle 
        [x^{\psi_i},\, y^{\psi_j}] 
        \mid x, y \in \{a,b\},\; x \neq y 
    \bigr\rangle
    = \langle [a^{\psi_i},\, b^{\psi_j}] \rangle,
\]
since $[a^{\psi_i}, a^{\psi_j}] = 1$ by the weak 
commutativity relation, and $[b^{\psi_i}, a^{\psi_j}] 
= [a^{\psi_i}, b^{\psi_j}]^{-1}$ by the symmetry $[x^{\psi_i}, y^{\psi_j}] = [x^{\psi_j}, y^{\psi_i}]$. 
For $L_{ij}$, recall from \cite[Section~4]{sidki-wp} 
that $L_{ij} = [H^{\psi_i}, \psi_i^{-1}\psi_j]$ is 
generated by elements $h^{-\psi_i}h^{\psi_j}$ for 
$h \in H$, together with the generators of 
$D_{ij} \leq L_{ij}$; since $H = \langle a, b\rangle$, 
it suffices to take $h \in \{a, b\}$, giving the 
stated description of $L_{ij}$.

We now establish the order claims. For the generator 
$d := [a^{\psi_i}, b^{\psi_j}]$ of $D_{ij}$: by
\cite[Proposition~4.1.6(i)]{sidki-wp}, since $H$ is a $2$-generated abelian group, $d$ is central
in $\langle H^{\psi_i}, H^{\psi_j}\rangle$, so the 
identity $[x^k, y] = [x,y]^k$, valid for central 
commutators, gives
\[
    d^{p^t} = [(a^{p^t})^{\psi_i},\, b^{\psi_j}] = 1
    \qquad\text{and}\qquad
    d^{p^u} = [a^{\psi_i},\, (b^{p^u})^{\psi_j}] = 1,
\]
since $a^{p^t} = 1$ and $b^{p^u} = 1$ in $H$. Hence 
$\mathrm{ord}(d)$ divides $\gcd(p^t, p^u) = p^u$.

For the generators of $L_{ij}$: the element 
$a^{-\psi_i}a^{\psi_j}$ has order dividing 
$\mathrm{ord}(a) = p^t$, since
\[
    (a^{-\psi_i}a^{\psi_j})^{p^t} 
    = (a^{p^t})^{-\psi_i}(a^{p^t})^{\psi_j} = 1,
\]
and $b^{-\psi_i}b^{\psi_j}$ has order dividing 
$\mathrm{ord}(b) = p^u$, since
\[
    (b^{-\psi_i}b^{\psi_j})^{p^u} 
    = (b^{p^u})^{-\psi_i}(b^{p^u})^{\psi_j} = 1,
\]
where both expansions are valid since $a^{\psi_i}$ and $a^{\psi_j}$ (resp.\ $b^{\psi_i}$ and $b^{\psi_j}$) commute, by weak commutativity. In particular, 
all generators of $L_{ij}$ have order dividing 
$\exp(H) = p^t$.

\medskip
\textbf{(ii).} We proceed in four steps.

\medskip
\textit{Step 1: $K$ is generated, as a normal subgroup, by six double
commutators.}
For each pair $0 \leq i < j \leq 2$ let $l$ be the remaining index of
$\{0,1,2\}$ and put $d_{ij} := [a^{\psi_i}, b^{\psi_j}]$, so that
$D_{ij} = \langle d_{ij} \rangle$ by~(i). We show that
\begin{equation}\label{eq:K_six}
    K = N := \bigl\langle\, [d_{ij}, a^{\psi_l}],\; [d_{ij}, b^{\psi_l}]
    \;:\; 0 \leq i < j \leq 2 \,\bigr\rangle^{\chi_3(H)} .
\end{equation}
Explicitly, the six generators of $N$ are
\begin{align*}
    &\{0,1\},\ l = 2: & [d_{01}, \app] &= [a, \bp, \app] = k_1, &
    [d_{01}, \bpp] &= [a, \bp, \bpp] = k_2,\\
    &\{0,2\},\ l = 1: & [d_{02}, \ap] &= [a, \bpp, \ap] = k_3, &
    [d_{02}, \bp] &= [a, \bpp, \bp],\\
    &\{1,2\},\ l = 0: & [d_{12}, a] &= [\ap, \bpp, a] = k_4, &
    [d_{12}, b] &= [\ap, \bpp, b].
\end{align*}

\emph{$N \subseteq K$.} Let $h \in \{a, b\}$. Since $\{i,l\} \neq \{i,j\}$,
the relation $[D_{ij}, L_{il}] = 1$ holds in $\chi(3,H)$, and
$h^{-\psi_i}h^{\psi_l} \in L_{il}$. As $d_{ij}$ is central in
$\langle H^{\psi_i}, H^{\psi_j}\rangle$ (proof of~(i)), we get, modulo $K$,
\[
    d_{ij}^{\,h^{\psi_l}} = \bigl(d_{ij}^{\,h^{\psi_i}}\bigr)^{h^{-\psi_i}h^{\psi_l}}
    = d_{ij}^{\,h^{-\psi_i}h^{\psi_l}} \equiv d_{ij},
\]
that is, $[d_{ij}, h^{\psi_l}] \in K$.

\emph{$K \subseteq N$.} By definition, $K$ is the normal closure in
$\chi_3(H)$ of the commutators $[x, y]$ with $x \in D_{ij}$, $y \in L_{kl}$
and $\{i,j\} \neq \{k,l\}$. Since $D_{ij} = \langle d_{ij}\rangle$ and, by
$[xz, y] = [x,y]^z[z,y]$, each $[d_{ij}^{\,m}, y]$ is a product of conjugates
of $[d_{ij}, y]^{\pm 1}$, it suffices to show that $[d_{ij}, w] \in N$ for
every $w \in \chi_3(H)$. Write $w$ as a word in the generators
$a^{\psi_r}, b^{\psi_r}$ ($r = 0,1,2$). The identities
\[
    [x, yz] = [x, z]\,[x, y]^{z}, \qquad
    [x, y^{-1}] = \bigl([x, y]^{-1}\bigr)^{y^{-1}}
\]
express $[d_{ij}, w]$ as a product of conjugates of elements
$[d_{ij}, g]^{\pm 1}$, with $g$ a generator. If $g \in H^{\psi_i} \cup
H^{\psi_j}$, then $[d_{ij}, g] = 1$ by the centrality of $d_{ij}$ in
$\langle H^{\psi_i}, H^{\psi_j}\rangle$; otherwise $g \in \{a^{\psi_l},
b^{\psi_l}\}$ and $[d_{ij}, g]$ is a generator of $N$. Hence
$[d_{ij}, w] \in N$, which proves~\eqref{eq:K_six}.

In particular, the commutators $[D_{ij}, D_{kl}]$ need no separate
treatment. For instance, with $d = d_{01}$,
\[
    [d, [\ap, \bpp]] = [d,\, a^{-\psi} b^{-\psi_2} a^{\psi}\, \bpp]
    = [a, \bp, \bpp]\,[d,\, a^{-\psi} b^{-\psi_2} a^{\psi}]^{\bpp}
    = k_2 \,\bigl(k_2^{-1}\bigr)^{b^{-\psi_2} a^{\psi} b^{\psi_2}} .
\]

\medskip
\textit{Step 2: Reduction to four generators (general $p$).}
We show that the two extra generators
$[a, \bpp, \bp]$ and $[\ap, \bpp, b]$ of Step~1
are conjugates of $k_2$ in $\chi_3(H)$. The argument
uses only weak commutativity, the Hall--Witt identity
\[
    [x, y^{-1}, z]^{y}\cdot[y, z^{-1}, x]^{z}
    \cdot[z, x^{-1}, y]^{x} = 1
\]
together with the fact that, by
\cite[Proposition~4.1.6(i)]{sidki-wp}, since $H$ is a $2$-generated abelian group, each $D_{ij}$
is central in $\langle H^{\psi_i}, H^{\psi_j}\rangle$
(no centrality of $K$ in $\chi_3(H)$ is assumed at
this stage).

\smallskip
\emph{Claim:
$[a, \bpp, \bp] = k_2^{\,[a, \bp]^{-1}\bp\, a^{-1}}$.}
Apply Hall--Witt with $(x, y, z) = (a, \bp, \bpp)$:
\begin{itemize}
    \item The middle term $[\bp, \bpp^{-1}, a]^{\bpp}$
    vanishes since $[\bp, \bpp^{-1}] = 1$ by weak
    commutativity between $\bp$ and $\bpp$.
    \item For the third term, since
    $[a, \bpp] \in D_{02}$ is central in
    $\langle H, H^{\psi_2}\rangle$ and
    $a \in H \subseteq \langle H, H^{\psi_2}\rangle$,
    \[
        [\bpp, a^{-1}]
        = [a, \bpp]^{a^{-1}} = [a, \bpp];
    \]
    hence
    $[\bpp, a^{-1}, \bp]^{a}
        = [\,[a, \bpp], \bp\,]^{a}
        = [a, \bpp, \bp]^{a}$.
    \item For the first term, since
    $[a, \bp] \in D_{01}$ is central in
    $\langle H, H^{\psi}\rangle$,
    \[
        [a, \bp^{-1}]
        = [a, \bp]^{-\bp^{-1}} = [a, \bp]^{-1};
    \]
    hence
    \[ [a, \bp^{-1}, \bpp]^{\bp}
        = [\,[a, \bp]^{-1}, \bpp\,]^{\bp}
        = \bigl(k_2^{-[a, \bp]^{-1}}\bigr)^{\bp}, \]
    using $[X^{-1}, Y] = [X, Y]^{-X^{-1}}$.
\end{itemize}
Hall--Witt therefore reads
\[
    \bigl(k_2^{-1}\bigr)^{[a, \bp]^{-1}\bp}\,[a, \bpp, \bp]^{a} = 1,
\]
and conjugating by $a^{-1}$ gives
\begin{equation}\label{eq:k5_conj}
    [a, \bpp, \bp]
    = k_2^{\,[a, \bp]^{-1}\bp\, a^{-1}}.
\end{equation}

\smallskip
\emph{Claim:
$[\ap, \bpp, b] = k_2^{\,\ap\, \bpp^{-1}\, [\ap, \bpp]}$.}
Apply Hall--Witt with $(x, y, z) = (b, \ap, \bpp)$:
\begin{itemize}
    \item The third term $[\bpp, b^{-1}, \ap]^{b}$
    vanishes since $[\bpp, b^{-1}] = 1$ by weak
    commutativity between $b$ and $\bpp$.
    \item For the first term, the basic identity
    $[a, \bp] = [\ap, b]$ (which holds in $\chi_3(H)$
    via the embedding $\chi(H)\hookrightarrow\chi_3(H)$)
    together with the centrality of $D_{01}$ in $\langle H, H^{\psi}\rangle$ give
    $[b, \ap^{-1}] = [\ap, b]^{\ap^{-1}} = [\ap, b]
    = [a, \bp]$; hence
    $[b, \ap^{-1}, \bpp]^{\ap}
        = [\,[a, \bp], \bpp\,]^{\ap}
        = k_2^{\ap}$.
    \item For the middle term, since
    $[\ap, \bpp] \in D_{12}$ is central in
    $\langle H^{\psi}, H^{\psi_2}\rangle$,
    \[
        [\ap, \bpp^{-1}]
        = [\ap, \bpp]^{-\bpp^{-1}} = [\ap, \bpp]^{-1};
    \]
    hence $[\ap, \bpp^{-1}, b]^{\bpp}
    = \bigl([\ap, \bpp, b]^{-[\ap, \bpp]^{-1}}\bigr)^{\bpp}$.
\end{itemize}
Hall--Witt then yields
\begin{equation}\label{eq:k6_conj}
    [\ap, \bpp, b]
    = k_2^{\,\ap\, \bpp^{-1}\, [\ap, \bpp]}.
\end{equation}

\smallskip
By~\eqref{eq:k5_conj} and~\eqref{eq:k6_conj}, the
elements $[a, \bpp, \bp]$ and $[\ap, \bpp, b]$ lie
in the normal closure of $\{k_2\}$ in $\chi_3(H)$.
Combined with Step~1, this proves that
\[
    K \;=\; \bigl\langle k_1, k_2, k_3, k_4
    \bigr\rangle^{\chi_3(H)}.
\]

\medskip
In Steps~3--5, which prove~(iii), $p$ is odd. Besides weak commutativity,
we use that each $D_{ij} = \langle d_{ij}\rangle$ is central in
$\langle H^{\psi_i}, H^{\psi_j}\rangle$ (so commutators between these two
copies are bilinear), the symmetry
$[x^{\psi_i}, y^{\psi_j}] = [x^{\psi_j}, y^{\psi_i}]$
(Lemma~\ref{lem:relacoes_chi}(i), whose proof uses only weak
commutativity), and the action of $S_3$ on $\chi_3(H)$ permuting the
copies (Lemma~\ref{lem:sym_aut}).

\medskip
\textit{Step 3: $[D_{ij}, D_{kl}] = 1$.}
Let $X := d_{02}$, $Y := d_{12}$ and $g := a^{-1}\ap$. As $X^a = X$ and
$Y^{\ap} = Y$, expanding $[a\ap, \bpp] = [\ap a, \bpp]$ and
$[\bpp^{-1}, a\ap] = [\bpp^{-1}, \ap a]$ gives
$(XY)^{g} = YX$ and $(YX)^{g} = XY$, so $(XY)^{g^2} = XY$. Since $g$ has
order dividing $p^t$, which is odd, $g \in \langle g^2\rangle$; hence
$XY = (XY)^g = YX$ and $[D_{02}, D_{12}] = 1$. The transpositions
$(0\;2)$ and $(1\;2)$ of the copies give the other two cases.

\medskip
\textit{Step 4: $k_1 = k_3 = k_4$ and $k_2 = [a, \bpp, \bp] = [\ap, \bpp, b]$.}
For $u, v \in H$ and $\{i, j, l\} = \{0, 1, 2\}$, put
$e := [u^{\psi_i}, v^{\psi_j}]$ and $f := [u^{\psi_l}, v^{\psi_j}]$. Then
\begin{equation}\label{eq:swap_u}
    [e, u^{\psi_l}] = [f, u^{\psi_i}] .
\end{equation}
Indeed, in Hall--Witt for $(u^{\psi_i}, v^{-\psi_j}, u^{\psi_l})$ the
third factor is trivial, the first equals
$[e, u^{\psi_l} f^{-1}] = [e, u^{\psi_l}]^{f^{-1}}$ (Step~3), and the
second equals $[f^{-1}, u^{\psi_i}] = ([f, u^{\psi_i}]^{-1})^{f^{-1}}$.
For $u = a$, $v = b$, \eqref{eq:swap_u} gives $k_1 = [d_{12}, a] = k_4$
if $(i,j,l) = (0,1,2)$ (as $[\app, \bp] = d_{12}$), and $k_3 = k_4$ if
$(i,j,l) = (0,2,1)$. For $k_2$: since $a^{\bpp} = a\,d_{02}$ and
$\bp^{\bpp} = \bp$, Step~3 gives $d_{01}^{\,\bpp} = d_{01}[d_{02}, \bp]$,
that is, $k_2 = [d_{02}, \bp]$; writing $d_{01} = [\ap, b]$, the same
computation gives $k_2 = [d_{12}, b]$.

\medskip
\textit{Step 5: $K \leq Z(\chi_3(H))$.}
If $c \in D_{ij}$, $w \in H^{\psi_l}$ and
$h \in H^{\psi_i} \cup H^{\psi_j}$, then $[c, w, h] = 1$: in Hall--Witt
for $(c, w^{-1}, h)$, the second factor vanishes by Step~3 and the third
by the centrality of $c$. Since $k_1 = [d_{01}, \app] = [d_{12}, a]$ and
$k_2 = [d_{01}, \bpp] = [d_{12}, b]$ (Step~4), $k_1$ and $k_2$ commute
with all three copies. By Steps~1 and~4 the six generators of $N$ reduce
to $k_1, k_2$, so $K = \langle k_1, k_2\rangle$ is central. (Step~2 is
thus not needed for $p$ odd.)

\medskip
\textbf{(iii).} This is Steps~4 and~5, together with the exponent: as
$k_1$ is central, $k_1^{p^u} = [d_{01}^{\,p^u}, \app] = 1$ by~(i), and
likewise $k_2^{p^u} = 1$.

\medskip
\textbf{(iv).} For $H = C_{2^t} \times C_{2^u}$ with
$t \geq u \geq 1$, the projection $H \twoheadrightarrow
H/H^2 \cong C_2 \times C_2$ induces an epimorphism
$\chi_3(H) \twoheadrightarrow \chi_3(C_2 \times C_2)$
that sends each $k_i$ to the corresponding generator of
$K$ in the Klein case. Composing with the matrix
representation of Theorem~\ref{thm:chi_n_klein}
therefore yields a homomorphism $\chi_3(H) \to
\mathrm{GL}_5(\mathbb{Z})$ that we use throughout.

The non-centrality of $K$ for $p = 2$ is
witnessed by the action $k_1^{b} = k_3^{-1} \neq k_1$
in the case $H = C_2\times C_2$: Hall--Witt applied to the involutions $(\ap, b, \app)$ gives $[d_{01},\app]^{b}\,[d_{02},\ap]^{\app} = 1$, the third factor being trivial as $[\app,\ap]=1$; since $k_3$ is fixed by $\app$, this is $k_1^{b} = k_3^{-1}$. Moreover $k_3^{-1} \neq k_1$, since under the representation $\rho$ of Theorem~\ref{thm:chi_n_klein} they correspond to the translation vectors $-v^{(3)} = (0,2,-4,0)$ and $v^{(1)} = (4,2,0,0)$.
The four generators $k_1, k_2, k_3, k_4$ are linearly
independent in $K^{\mathrm{ab}}$: the matrix
representation in the proof of
Theorem~\ref{thm:chi_n_klein} maps them to four
linearly independent vectors of
$\mathbb{Z}^4 \leq K^{\mathrm{ab}}$, so in particular
$K^{\mathrm{ab}}$ has $\mathbb{Z}$-rank at least $4$.
Consequently, the equalities $k_1 = k_3 = k_4$ and
$k_2 = [a, \bpp, \bp] = [\ap, \bpp, b]$ used in the
$p$ odd reduction (Step~4) fail for $p = 2$, so that
reduction does not apply. The generators $k_i$ may
have infinite order, again by the matrix
representation.
\end{proof}

We are now ready to prove Theorem~\hyperref[th:thmD]{D}(i), restated as follows.

\begin{theorem}\label{thm:chi3cpmpn}
Let $p$ be an odd prime and $H \cong C_{p^t} \times C_{p^u}$
with $t \geq u \geq 1$. Then $\chi_3(H)$ is finite of
order dividing $p^{3t+8u}$ and has nilpotency class
at most $3$.
\end{theorem}

\begin{proof}
By Lemma~\ref{lem:kernel_general}(ii)--(iii), $K$ is 
central in $\chi_3(H)$ and
\[
    K = \langle k_1,\, k_2 \rangle,
\]
with $\exp(K)$ dividing $p^u$. Hence $|K|$ divides $p^{2u}$.

Since $K \leq Z(\chi_3(H))$ and $\chi(3,H)$ has nilpotency 
class at most $2$ by Corollary~\ref{cor:abelian_odd} and 
\cite[Theorem~4.2.4]{sidki-wp}, it follows that $\chi_3(H)$ 
is nilpotent of class at most $3$.

For the order bound: by 
\cite[Theorem~4.2.4~(i)--(ii)]{sidki-wp}, for $H$ of type 
$(p^t, p^u)$ we have $|D| = p^u$, hence 
$|\chi(H)| = |D||H|^2 = p^{2t+3u}$. By 
Theorem~\ref{th:upperbound} and 
Corollary~\ref{cor:abelian_odd},
\[
    |\chi(3,H)| = |D|^3 \cdot |H|^3 = p^{3u} \cdot p^{3(t+u)} 
    = p^{3t+6u}.
\]
Since $|\chi_3(H)| = |K| \cdot |\chi(3,H)|$, we conclude 
that
\[
    |\chi_3(H)| \;\text{ divides }\; 
    p^{2u} \cdot p^{3t+6u} = p^{3t+8u}.
\]
In particular, when $t = u = 1$ this gives the bound $p^{11}$.
\end{proof}

We now establish Theorem~\hyperref[th:thmD]{D}(ii), restated below.

\begin{theorem}\label{thm:chi_n_klein}
Let $H$ be the Klein four-group. Then $\chi_3(H)$ is an
infinite group. Moreover, the kernel of
$\pi : \chi_3(H) \twoheadrightarrow \chi(3,H)$ has
$\mathbb{Z}^4$ as a homomorphic image.
\end{theorem}

\begin{proof}
As above, we assume that $H$ is generated by $a$ and $b$. According to Lemma~\ref{lem:kernel_general}(ii), $K = \langle k_1, k_2, k_3, k_4\rangle^{\chi_3(H)}$.
We exhibit a homomorphism $\rho \colon \chi_3(H) \to \mathrm{GL}_5(\mathbb{Z})$
with $\rho(K) \cong \mathbb{Z}^4$. For this purpose, let's consider the assignments 
\[
a \mapsto A:=\begin{pmatrix}
 1 &  1 &  0 &  0 & 0 \\
 0 & -1 &  0 &  0 & 0 \\
 0 & -1 &  1 &  0 & 0 \\
 0 &  0 &  0 & -1 & 0 \\
 0 &  1 &  0 &  0 & 1
\end{pmatrix},
\quad
b \mapsto B := \begin{pmatrix}
  0 &  0 & -1 &  1 & 0 \\
  0 &  1 &  0 & -2 & 0 \\
 -1 &  0 &  0 & -1 & 0 \\
  0 &  0 &  0 & -1 & 0 \\
  0 &  0 &  0 &  1 & 1
\end{pmatrix},
\]
\[
\ap \mapsto A^\psi := \begin{pmatrix}
 1 &  1 &  0 &  0 & 0 \\
 0 & -1 &  0 &  0 & 0 \\
 0 &  0 & -1 &  0 & 0 \\
 0 & -1 &  0 &  1 & 0 \\
 0 &  0 &  2 &  0 & 1
\end{pmatrix},
\quad
\bp \mapsto  B^\psi := \begin{pmatrix}
 0 &  0 &  1 & -1 & 0 \\
 0 &  1 & -2 &  0 & 0 \\
 0 &  0 & -1 &  0 & 0 \\
-1 &  0 & -1 &  0 & 0 \\
-1 &  0 &  2 & -1 & 1
\end{pmatrix},
\]
\[
\app \mapsto A^{\psi_2} := \begin{pmatrix}
-1 &  0 &  0 &  0 & 0 \\
 0 & -1 &  0 &  0 & 0 \\
 0 & -1 &  1 &  0 & 0 \\
 0 & -1 &  0 &  1 & 0 \\
 0 &  1 &  0 &  0 & 1
\end{pmatrix},
\quad
\bpp \mapsto B^{\psi_2}:= \begin{pmatrix}
 1 &  0 &  0 &  0 & 0 \\
-2 & -1 &  0 &  0 & 0 \\
-1 &  0 &  0 & -1 & 0 \\
-1 &  0 & -1 &  0 & 0 \\
 1 &  1 &  0 &  0 & 1
\end{pmatrix}.
\]
Each of these matrices is an involution, $A$ commutes with $B$, $A^\psi$ with
$B^\psi$ and $A^{\psi_2}$ with $B^{\psi_2}$, and, for $M \in \{A, B, AB\}$, the
three corresponding matrices $M, M^\psi, M^{\psi_2}$ commute pairwise. These are
the defining relations of $\chi_3(H)$, since $H = \{1, a, b, ab\}$, and checking
them is a direct computation (which can also be carried out in
GAP~\cite{gap}). By von Dyck's theorem the assignment extends to a homomorphism
$\rho$.

All six matrices have last column $(0,0,0,0,1)^{\top}$, so $\rho(\chi_3(H))$
lies in the group of integer matrices
$\left(\begin{smallmatrix} M & \mathbf{0} \\ w^{\top} & 1 \end{smallmatrix}\right)$,
in which
\[
    T := \Bigl\{\, T(v) := \begin{pmatrix} I_4 & \mathbf{0} \\ v^{\top} & 1 \end{pmatrix}
    : v \in \mathbb{Z}^4 \Bigr\}
\]
is a normal subgroup, isomorphic to $\mathbb{Z}^4$ via $T(v)\,T(w) = T(v + w)$.
A direct computation gives $\rho(k_s) = T(v^{(s)})$, $s = 1, \ldots, 4$, with
\[
\begin{pmatrix} v^{(1)} \\ v^{(2)} \\ v^{(3)} \\ v^{(4)} \end{pmatrix}
= \begin{pmatrix}
4 & 2 & 0 & 0 \\
2 & 4 & -2 & -2 \\
0 & -2 & 4 & 0 \\
0 & -2 & 0 & 4
\end{pmatrix},
\qquad \det = 64 \neq 0 .
\]
Since $T$ is normal, $\rho(K)$, the normal closure of
$\rho(k_1), \ldots, \rho(k_4)$, lies in $T$; and it contains the subgroup of
rank~$4$ spanned by $v^{(1)}, \ldots, v^{(4)}$. Hence $\rho(K) \cong \mathbb{Z}^4$.
Therefore $K$, and so $\chi_3(H)$, is infinite, and $K$ maps onto $\mathbb{Z}^4$.

\end{proof}
\section{The Hybrid Operator $\chi_4 - \chi^2$}\label{sec:hybrid}

Both the iterated group $\chi^{2}(H)$ and the mother group $\chi_{4}(H)$ are generated by four isomorphic copies of $H$, but by different sets of defining relations. In this section we compare the two constructions. We determine their largest common quotient, which we denote by $Q(H)$ and call the \emph{hybrid} of $\chi_{4}$ and $\chi^{2}$.

\subsection{The four copies of $H$ inside $\chi^{2}(H)$}

Write $K=\chi(H)=\langle H, H^{\psi}\rangle$. Applying the operator $\chi$ once more, let $\psi': K \to K^{\psi'}$ denote the isomorphism onto the second copy of $K$, so that
\begin{equation*}
\chi^{2}(H)=\chi(K)=\bigl\langle\, K, K^{\psi'} \;\bigm|\; [x,x^{\psi'}]=1 \ \text{for all} \ x \in K \,\bigr\rangle .
\end{equation*}
Then $\chi^{2}(H)$ is generated by the four subgroups
\begin{equation*}
H, \qquad H^{\psi}, \qquad H^{\psi'}, \qquad H^{\psi\psi'},
\end{equation*}
each isomorphic to $H$. Throughout this section we match them with the four copies of $\chi_{4}(H)$ by
\begin{equation*}
H^{\psi_{0}} \leftrightarrow H, \qquad
H^{\psi_{1}} \leftrightarrow H^{\psi}, \qquad
H^{\psi_{2}} \leftrightarrow H^{\psi'}, \qquad
H^{\psi_{3}} \leftrightarrow H^{\psi\psi'}.
\end{equation*}

\begin{lemma}\label{lem:hyb_copias}
In $\chi^{2}(H)$ the following relations hold for every $h \in H$:
\begin{equation*}
[h,h^{\psi}]=1, \qquad [h^{\psi'},h^{\psi\psi'}]=1, \qquad [h,h^{\psi'}]=1, \qquad [h^{\psi},h^{\psi\psi'}]=1 .
\end{equation*}
\end{lemma}
\begin{proof}
The first is a defining relation of $K=\chi(H)$ and the second is its image under $\psi'$. The third and the fourth are the defining relation $[x,x^{\psi'}]=1$ of $\chi(K)$ applied to $x=h$ and to $x=h^{\psi}$ respectively.
\end{proof}

Thus, of the six relations $[h^{\psi_{i}},h^{\psi_{j}}]=1$ that define $\chi_{4}(H)$, exactly four hold in $\chi^{2}(H)$ by construction. The two remaining ones, corresponding to the pairs $\{0,3\}$ and $\{1,2\}$, are
\begin{equation*}
[h,h^{\psi\psi'}]=1 \qquad \text{and} \qquad [h^{\psi},h^{\psi'}]=1,
\end{equation*}
and they need not hold. Conversely, $\chi^{2}(H)$ satisfies the relations $[x,x^{\psi'}]=1$ for every $x \in K$, and these are not relations of $\chi_{4}(H)$. Hence neither of the two defining sets of relations contains the other.

\subsection{A reduced hybrid group} \label{ssec:hybrid}

\begin{definition}\label{def:hyb_Q}
Let $N_{0}$ be the normal closure in $\chi^{2}(H)$ of the set
\begin{equation*}
\bigl\{\, [h,h^{\psi\psi'}], \ [h^{\psi},h^{\psi'}] \;:\; h \in H \,\bigr\},
\end{equation*}
and put $Q(H)=\chi^{2}(H)/N_{0}$.
\end{definition}

That is, $Q(H)$ is the quotient of $\chi^{2}(H)$ by the relations of $\chi_{4}(H)$ which are missing from it.

\begin{lemma}\label{lem:hyb_ident}
 The following identity holds in $\chi^{2}(H)$ for all $h,k \in H$.
\begin{equation*}
[h,k^{\psi\psi'}]^{\,k^{\psi}} \cdot [k^{\psi},h^{\psi'}]^{\,k^{\psi\psi'}} = 1 .
\end{equation*}
\end{lemma}
\begin{proof}
Apply the defining relation $[x,x^{\psi'}]=1$ of $\chi(K)$ to the element $x=h\,k^{\psi}$ of $K$, so that $x^{\psi'}=h^{\psi'}k^{\psi\psi'}$, and expand:
\begin{equation*}
[h k^{\psi},\, h^{\psi'}k^{\psi\psi'}] = [h,\, h^{\psi'}k^{\psi\psi'}]^{\,k^{\psi}} \; [k^{\psi},\, h^{\psi'}k^{\psi\psi'}] .
\end{equation*}
By Lemma~\ref{lem:hyb_copias} we have $[h,h^{\psi'}]=1$ and $[k^{\psi},k^{\psi\psi'}]=1$, whence $[h,h^{\psi'}k^{\psi\psi'}]=[h,k^{\psi\psi'}]$ and $[k^{\psi},h^{\psi'}k^{\psi\psi'}]=[k^{\psi},h^{\psi'}]^{\,k^{\psi\psi'}}$.
\end{proof}

Taking $h=k$ in Lemma~\ref{lem:hyb_ident} shows that the two missing relations determine one another: killing either one of them kills the other. For $h \neq k$ the identity relates $D_{03}$ to $D_{12}$.

Let $A=\langle H^{\psi_{0}}, H^{\psi_{1}}\rangle$ and $B=\langle H^{\psi_{2}}, H^{\psi_{3}}\rangle$ be the two halves of $\chi_{4}(H)$, and let $\tau: A \to B$ be the isomorphism determined by $h^{\psi_{0}} \mapsto h^{\psi_{2}}$ and $h^{\psi_{1}} \mapsto h^{\psi_{3}}$. It exists because it is the restriction of the automorphism $\alpha_{\sigma}$ of $\chi_{4}(H)$ associated with the permutation $\sigma=(0\;2)(1\;3)$ of the copies: the argument of Lemma~\ref{lem:sym_aut} applies verbatim to $\chi_{n}(H)$, whose defining relations involve only weak commutativity.

\begin{proposition}\label{prop:hyb_Q}
The assignment $h^{\psi_{0}} \mapsto h$, $h^{\psi_{1}} \mapsto h^{\psi}$, $h^{\psi_{2}} \mapsto h^{\psi'}$, $h^{\psi_{3}} \mapsto h^{\psi\psi'}$ extends to an epimorphism $\pi: \chi_{4}(H) \twoheadrightarrow Q(H)$, and
\begin{equation*}
\ker \pi = \bigl\langle\!\bigl\langle \; [x,x^{\tau}] \;:\; x \in A \; \bigr\rangle\!\bigr\rangle .
\end{equation*}
Moreover $Q(H)$ is the largest common quotient of $\chi^{2}(H)$ and $\chi_{4}(H)$ compatible with the above matching of the copies.
\end{proposition}
\begin{proof}
By Lemma~\ref{lem:hyb_copias} four of the six defining relations of $\chi_{4}(H)$ hold in $\chi^{2}(H)$, and the two remaining ones hold in $Q(H)$ by Definition~\ref{def:hyb_Q}; hence $\pi$ is well defined, and it is surjective because the four copies generate. Both groups are quotients of the free product of the four copies of $H$; writing $R_{4}$ for the six families of weak commutativity relations and $R_{\chi^{2}}$ for the relations of $\chi^{2}(H)$, the above shows that $Q(H)$ is the quotient by the normal closure of $R_{\chi^{2}} \cup R_{4}$, which is the largest common quotient. Since the relations in $R_{4}$ already hold in $\chi_{4}(H)$, the kernel of $\pi$ is the normal closure of the image of $R_{\chi^{2}}$, that is, of the elements $[x,x^{\tau}]$ with $x \in A$.
\end{proof}

In other words, $Q(H)$ is the largest quotient of $\chi_{4}(H)$ in which the two halves $A$ and $B$ commute weakly with each other, and the kernel of $\pi$ is precisely the corresponding set of weak commutativity relations one level up.

\subsection{Diagonal representations and co-nilpotent subgroups}
The Diagonal representation of the mother group $\chi_n(H)$ introduced in Subsection~\ref{ssec:diag} suggests the following approach: consider a family $\mathcal{F}$ of groups $H$, and a larger group $G = G(\mathcal{F}, H, n)$ generated by $n$ isomorphic copies of $H$ that commute weakly with one another; that is, $G$ is a homomorphic image of $\chi_n(H)$, with $H \in \mathcal{F}$, where the $n$ copies of $H$ are kept intact. Given a positive integer $n$, the direct product $\mathbf{X}_m(H)$, with $m \geq n$, contains subgroups generated by $n$ copies of $H$, each defined by partial diagonals $(h_i)$, where for each $i$, $h_i = e$ or $h$ (see also Appendix~\ref{app:hamming} for details).
\begin{definition} \label{def:H-rank}
    A subgroup $K$ of a group $P$ is said to be \textit{co-nilpotent} provided the quotient of $P$ by the core $K_P$ is nilpotent.
\end{definition}
Let $\delta: G \to \mathbf{X}_m(H)$ be a diagonal representation of $G$ where its image is co-nilpotent in the direct product. With $n$ and $\mathcal{F}$ fixed, let $\mathbf{M}(n, \mathcal{F})$ be the largest $m$ which holds for all $H$ in $\mathcal{F}$.

\subsubsection{\bf Examples} 
    \begin{enumerate}
        \item As seen in Subsection~\ref{ssec:diag}, the representation $\delta_n$ of the mother group $\chi_n(H)$ gives that $\mathbf{M}(n, \mathcal{F}) = 2^n-1$ for any family $\mathcal{F}$; 
        \item The representation $\lambda_{\tilde{H}}: \chi(n,H) \to T_n({\tilde{H}})$ in Proposition~\ref{prop:epichintn} shows that $\mathbf{M}(n, \mathcal{F}) = n+1$ for these groups;
        \item The hybrid group $Q(H)$. The fact that both $H$ and $H^{\psi^{\prime}}$ embed in $Q(H)$ implies that the diagonal representation $Q(H) \to \mathbf{X}_6(H)$ induced from that of $\chi_4(H)$ is well-defined. Consequently, we get for $Q(H)$, $6 \leq \mathbf{M}(4,\mathcal{F}).$ Furthermore, applying $\delta_2$ twice, $\chi^2(H)=\chi(\chi(H)) \to \mathbf{X}_3(\chi(H)) \to \mathbf{X}_9(H)$, we obtain a diagonal representation of $\chi^2(H)$ with image $T_2(T_2(H))$. Since $T_2(T_2(H)) \geq \mathbf{X}_3\bigl(T_2(H)'\bigr) = \mathbf{X}_9(H')$, this image is co-nilpotent in $\mathbf{X}_9(H)$, and it follows that $\mathbf{M}(4,\mathcal{F}) \leq 9.$
    \end{enumerate}
We conclude this work with the following question:
\subsubsection{\bf Question} Does there exist a family $\mathcal{F}$ of finite groups for which $\mathbf{M}(n, \mathcal{F})$ has exponential growth with respect to $n$?

\appendix
\section{Weak commuting groups and Hamming codes}\label{app:hamming}
\begin{flushright}\footnotesize by Dmytro Savchuk\footnote{University of South Florida, 4202 E Fowler Ave, Tampa, FL 33620-5700, USA. \texttt{savchuk@usf.edu}. Partially supported by NSF RTG grant DMS-2342254.}\end{flushright}

\providecommand{\Z}{\mathbb{Z}}
\providecommand{\be}{\mathbf{e}}

For each group $H$ we consider $n$ isomorphic copies $H_1=H,H_2,\ldots, H_n$ of $H$. For each choice of isomorphisms $\psi_i\colon H\to H_i$ with $\psi_1$ being the identity isomorphism, we define a group
\[\chi_n(H)=\langle H_i,i=1,\ldots,n\mid [h,\psi_i(h)], h\in H, i=1,\ldots,n\rangle.\]
For $h\in H$ we will denote $\psi_i(h)$ simply by $h_i$ where it does not lead to any confusion.

The group $\chi_2(H)$ was defined in~\cite{sidki-wp} and also studied in~\cite{KS}. In particular, it has been shown that the homomorphism
\[\rho_2\colon \chi_2(H)\to H\times H\times H\]
induced by
\[\begin{array}{lcl}
h_1&\mapsto&(h,h,1)\\
h_2&\mapsto&(1,h,h)
\end{array}
\]
satisfies $H'\times H'\times H'<\rho_2(\chi_2(H))$. Clearly, in the case when $H$ is perfect, $\rho_2$ is surjective. The following proposition extends the above result to $\chi_n(H)$ using the parity check matrix of the Hamming code~\cite{hamming1950error} (see, for example,~\cite{TrappeWashington2020} for a gentle introduction to coding theory containing the discussion of Hamming codes and associated matrices).

\begin{theorem}
There is a homomorphism $\rho_n\colon\chi_n(H)\to H^{2^n-1}$ whose image contains $\gamma_n(H)^{2^n-1}$.
\end{theorem}

\begin{proof}
The parity check matrix for the binary Hamming code of length $2^n-1$ is the $n\times (2^{n}-1)$ matrix whose columns are all nonzero vectors in $\Z_2^n$. More precisely, we define $M_1=[1]$ and
\[M_{n+1}=\left[
\begin{array}{ccc|c|ccc}
\multicolumn{3}{c|}{\multirow{3}{*}{$M_n$}}&0&\multicolumn{3}{c}{\multirow{3}{*}{$M_n$}}\\
\multicolumn{3}{c|}{}&\vdots&\multicolumn{3}{c}{}\\
\multicolumn{3}{c|}{}&0&\multicolumn{3}{c}{}\\ \hline
0&\cdots&0&1&1&\cdots&1
\end{array}\right]
\]

For each $h\in H$ we will define $\rho_n(h_i)$ to be a $(2^n-1)$-tuple from $H^{2^n-1}$ that is obtained from the $i$-th row of $M_n$ by replacing $1$'s with $h$ and $0$'s with $1\in H$. It is straightforward to verify that $\rho_n$ extends to a homomorphism from $\chi_n(H)$ to $H^{2^n-1}$.

Now we will prove the claim by induction on $n$. For $n=1$ the statement is obvious. Assume that the statement is true for some $n\geq1$. To simplify calculations below for each $h\in H$, $n\geq 1$, and $1\leq i\leq 2^n-1$ we denote by $\be_{h,i}^{(n)}$ the element of $H^{2^n-1}$ whose all coordinates except the $i$-th one are trivial, and the $i$-th coordinate is equal to $h$. I.e.,
\[\be_{h,i}^{(n)}=(\underbrace{1,1,\ldots,1,h,1,\ldots,1}_{2^n-1}).\]
Then by induction assumption for each $h\in\gamma_n(H)$ there is $k^{(n)}_{h,i}\in\chi_{n}(H)$ such that $\rho_n\bigl(k^{(n)}_{h,i}\bigr)=\be_{h,i}^{(n)}$.
Since $\chi_n(H)$ is naturally contained in $\chi_{n+1}(H)$, we can view $k^{(n)}_{h,i}$ also as an element of $\chi_{n+1}(H)$. By the inductive definition of matrix $M_{n+1}$ we have:
\begin{equation}
\rho_{n+1}\bigl(k^{(n)}_{h,i}\bigr)=\bigl(\be_{h,i}^{(n)},1,\be_{h,i}^{(n)}\bigr).
\end{equation}
For each $g\in H$ the image of $g_{n+1}\in\chi_{n+1}(H)$ under $\rho_{n+1}$, according to the last row of $M_{n+1}$, is
\begin{equation}
\rho_{n+1}(g_{n+1})=(\underbrace{1,\ldots,1}_{2^n-1},g,\underbrace{g,\ldots,g}_{2^n-1}).
\end{equation}
Therefore, for $k^{(n+1)}_{[h,g],2^n+i}:=[k^{(n)}_{h,i},g_{n+1}]\in\chi_{n+1}(H)$ we get
\begin{equation}
\rho_{n+1}\bigl(k^{(n+1)}_{[h,g],2^n+i}\bigr)=\bigl(\underbrace{1,\ldots,1}_{2^n-1},1,\be_{[h,g],i}^{(n)}\bigr)=\be_{[h,g],2^n+i}^{(n+1)}.
\end{equation}
Since the choices of $0\leq i\leq 2^n-1$, $h\in\gamma_n(H)$ and $g\in H$ were arbitrary, we obtain that
\begin{equation}
\label{eq:app-last}
\underbrace{1\times\cdots\times1}_{2^n}\times\underbrace{\gamma_{n+1}(H)\times\cdots\times\gamma_{n+1}(H)}_{2^n-1}<\rho_{n+1}(\chi_{n+1}(H)).
\end{equation}

To deal with the first $2^n-1$ coordinates, we consider for each $0\leq i\leq 2^n-1$ the element $k^{(n+1)}_{[h,g],i}:=k^{(n)}_{[h,g],i}\cdot \bigl(k^{(n+1)}_{[h,g],2^n+i}\bigr)^{-1}$ for which we have
\begin{equation}
\label{eq:app-first}
\rho_{n+1}\bigl(k^{(n+1)}_{[h,g],i}\bigr)=\bigl(\be_{[h,g],i}^{(n)},1,\underbrace{1,\ldots,1}_{2^n-1}\bigr)=\be_{[h,g],i}^{(n+1)}.
\end{equation}

Finally, for the $2^n$-th coordinate, we consider
\[k^{(n+1)}_{[h,g],2^n}=[h,g]_{2^n}\cdot\prod_{i=1}^{2^n-1}\bigl(k^{(n+1)}_{[h,g],2^n+i}\bigr)^{-1}\]
for which we have
\begin{multline}
\label{eq:app-middle}
\rho_{n+1}(l_{[h,g],2^n})=(\underbrace{1,\ldots,1}_{2^n-1},[h,g],\underbrace{[h,g],\ldots,[h,g]}_{2^n-1})\cdot(\underbrace{1,\ldots,1}_{2^n-1},1,\underbrace{[h,g]^{-1},\ldots,[h,g]^{-1}}_{2^n-1})\\
=(\underbrace{1,\ldots,1}_{2^n-1},[h,g],\underbrace{1,\ldots,1}_{2^n-1})=\be_{[h,g],2^n}^{(n+1)}.
\end{multline}
Equalities~\eqref{eq:app-last},\eqref{eq:app-middle}, and~\eqref{eq:app-first} complete the proof.
\end{proof}

\begin{corollary}
If $H$ is perfect then $\rho_n\colon\chi_n(H)\to H^{2^n-1}$ is surjective.
\end{corollary}

\end{document}